\documentclass{article}

\usepackage{a4wide}
\usepackage{amsmath}
\usepackage{amssymb}

\usepackage{framed}
\usepackage{amsthm}
\usepackage{fancybox}
\usepackage{graphicx}
\usepackage[noend]{algorithmic}
\usepackage{booktabs,multirow}
\usepackage{mathtools}
\usepackage{color}
\usepackage{dsfont}

\def\beq{\begin{equation}}
\def\eeq{\end{equation}}

\def\bea{\begin{eqnarray}}
\def\eea{\end{eqnarray}}

\def\bsp{\begin{split}}
\def\esp{\end{split}}

\def\am{\arg\min}

\def\ol{\overline}
\def\tfrac#1#2{{\textstyle \frac{#1}{#2}}}  

\def\cR{\mathcal{R}}
\def\cS{\mathcal{S}}

\def\N{{{\rm I}\!{\rm N}}}
\def\R{{{\rm I}\!{\rm R}}}

\def\be#1{\begin{equation} \label{#1} }
\def\ee{\end{equation}}
\def\barr{\begin{array}}
\def\earr{\end{array}}

\def\norm#1{\hspace{0.2ex} \left\|#1\right\|}

\def\ol#1{\overline{#1}}
\def\tfrac#1#2{{\textstyle \frac{#1}{#2}}}
\def\qdag{q^\dagger}

\def\L{{\cal L}}
\def\M{{\cal M}}

\newcommand{\normklein}[1]{\|{#1}\|}

\newcommand{\normQ}[1]{\norm{#1}_Q}

\newcommand{\normV}[1]{\norm{#1}_V} 
\newcommand{\normQklein}[1]{\normklein{#1}_Q}   
\newcommand{\normGklein}[1]{\normklein{#1}_G} 
 
\newcommand{\normVklein}[1]{\normklein{#1}_V} 
\newcommand{\normWklein}[1]{\normklein{#1}_W} 

\newcommand{\bF}{\mathbf{F}}

\newcommand{\bY}{\mathbf{Y}}

\newcommand{\normQV}[1]{\norm{#1}_{Q\times V}}      

 \newcommand{\idV}{\operatorname{id}}
 \newcommand{\idQ}{\operatorname{id}}
 \newcommand{\idQV}{\operatorname{id}}

\newcommand{\dualW}[2]{\langle#1,#2\rangle_{W^*,W}}    
\newcommand{\dualV}[2]{\langle#1,#2\rangle_{V^*,V}}    

\def\cR{\mathcal{R}}
\def\cS{\mathcal{S}}
\def\cT{\mathcal{T}}

\newcommand{\scalarQ}[2]{(#1,#2)_Q}                          
\newcommand{\scalarV}[2]{(#1,#2)_V}                          
\newcommand{\scalarG}[2]{(#1,#2)_G} 
\newcommand{\scalarWstern}[2]{(#1,#2)_{W^*}}

\newcommand{\normWstern}[1]{\norm{#1}_{W^*}}
\newcommand{\normWsternklein}[1]{\normklein{#1}_{W^*}}
\newcommand{\vektor}[2]{\begin{pmatrix}{#1}\\{#2}\end{pmatrix}}

\newcommand{\qold}{q_{\operatorname{old}}}
\newcommand{\uold}{u_{\operatorname{old}}}

\newcommand{\qkdelta}{q^{k,\delta}}

\newcommand{\qksterndelta}{q^{k_*,\delta}}

\newcommand{\uk}{u^k}

\newcommand{\ukdelta}{u^{k,\delta}}
\newcommand{\uksterndelta}{u^{k_*,\delta}}

\newcommand{\qh}{q_h}
\newcommand{\xh}{x_h}

\newcommand{\zh}{z_h}
\newcommand{\uh}{u_h}

\newcommand{\yh}{y_h}

\newcommand{\xeinsh}{x_h^1}

\newcommand{\qhzweik}{q^k_{h^2_k}}

\newcommand{\vhzweik}{v^k_{h^2_k}}

\newcommand{\zhkplus}{z_{h}^{k+1}}

\newcommand{\vhk}{v^k_h}
\newcommand{\uhk}{u^k_h}

\newcommand{\qoldnull}{q_{\operatorname{old}}^0}
\newcommand{\uoldnull}{u_{\operatorname{old}}^0}

\newcommand{\qhnullnull}{q^0_{h_0}}
\newcommand{\uhnullnull}{u^0_{h_0}}

\newcommand{\zhnullnull}{z^0_{h_0}}
\newcommand{\qhnull}{q^0_{h}}
\newcommand{\uhnull}{u^0_{h}}

\newcommand{\zhnull}{z^0_{h}}
\newcommand{\qhk}{q^k_h}
\newcommand{\qhkk}{q^k_{h_k}}
\newcommand{\uhkk}{u^k_{h_k}}

\newcommand{\pk}{p^k}

\newcommand{\qk}{q^k}
\newcommand{\zk}{z^k}

\newcommand{\qoldksterndeltal}{q_{\operatorname{old}}^{k_*(\delta_l)}}

\newcommand{\uoldksterndeltal}{u_{\operatorname{old}}^{k_*(\delta_l)}}
\newcommand{\qoldkstern}{q_{\operatorname{old}}^{k_*}}
\newcommand{\uoldkstern}{u_{\operatorname{old}}^{k_*}}

\newcommand{\qoldzwei}{q_{\operatorname{old}}^2}
\newcommand{\qolddrei}{q_{\operatorname{old}}^3}
\newcommand{\qoldvier}{q_{\operatorname{old}}^4}
\newcommand{\uoldk}{u_{\operatorname{old}}^k}
\newcommand{\uoldhk}{u_{\operatorname{old},h}^k}

\newcommand{\ukminuseins}{u^{k-1}}

\newcommand{\qoldk}{q_{\operatorname{old}}^k}
\newcommand{\qkminuseins}{q^{k-1}}

\newcommand{\uoldkplus}{u_{\operatorname{old}}^{k+1}}

\newcommand{\qoldkplus}{q_{\operatorname{old}}^{k+1}}

\newcommand{\qhkstern}{q_h^{k_*}}
\newcommand{\uhkstern}{u_h^{k_*}}

\newcommand{\qhkminuseinsk}{q^{k-1}_{h_{k-1}}}
\newcommand{\uhkminuseinsk}{u^{k-1}_{h_{k-1}}}

\newcommand{\Lk}{L_{k-1}}
\newcommand{\Kk}{K_{k-1}}
\newcommand{\Ck}{C_{k-1}}
\newcommand{\Lhk}{L_{h,k-1}}
\newcommand{\Khk}{K_{h,k-1}}
\newcommand{\Chk}{C_{h,k-1}}
\newcommand{\rg}{r^g_{k-1}}
\newcommand{\rf}{r^f_{k-1}}

\newcommand{\cadj}{c_{\operatorname{adj}}}

 \newcommand{\Ffett}{\mathbf{F}}
 
 \newcommand{\gfettdelta}{\mathbf{g}^\delta}
 \newcommand{\Ieinsh}{I_{1,h}}
 \newcommand{\Izweih}{I_{2,h}}
 \newcommand{\Idreih}{I_{3,h}}
 \newcommand{\Ivierh}{I_{4,h}}
 \newcommand{\Ieinshk}{I_{1,h}^k}
 \newcommand{\Izweihk}{I_{2,h}^k}
 \newcommand{\Idreihk}{I_{3,h}^k}
 \newcommand{\Ivierhk}{I_{4,h}^k}

\newcommand{\uttheta}{\tilde{\underline\theta}}
\newcommand{\ottheta}{\tilde{\overline \theta}}

\author{B.~Kaltenbacher \and A.~Kirchner \and B.~Vexler}
\title{Goal oriented adaptivity in the IRGNM for parameter identification in PDEs II:\\
all-at once formulations}

\newtheorem{cor}{Corollary}
\newtheorem{thm}{Theorem}

\newtheorem{lem}{Lemma}
\newtheorem{rem}{Remark}

\newtheorem{ass}{Assumption}
\newtheorem{algorithm}{Algorithm}{\bf}{\it}

\begin{document}

\maketitle

\abstract{
In this paper we investigate adaptive discretization of the iteratively regularized Gauss-Newton method IRGNM. All-at-once formulations considering the PDE and the measurement equation simultaneously allow to avoid (approximate) solution of a potentially nonlinear PDE in each Newton step as compared to the reduced form \cite{KKVV13}. We analyze a least squares and a generalized Gauss-Newton formulation and in both cases prove convergence and convergence rates with a posteriori choice of the regularization parameters in each Newton step and of the stopping index under certain accuracy requirements on four quantities of interest. Estimation of the error in these quantities by means of a weighted dual residual method is discussed, which leads to an algorithm for adaptive mesh refinement. Numerical experiments with an implementation of this algorithm show the numerical efficiency of this approach, which especially for strongly nonlinear PDEs outperforms the nonlinear Tikhonov regularization considered in \cite{KKV10}.
}

\section{Introduction}
We consider the problem of identifying a parameter $q$ in a PDE
\be{eq_PDE} 
A(q,u)=f
\ee
from measurements of the state $u$
\be{eq_meas}
C(u)=g\,,
\ee
where $q\in Q$, $u\in V$, $g\in G$, $Q,V,G$ are Hilbert spaces and $A\colon Q \times V\to W^*$
with $W^*$ denoting the dual space of some Hilbert space $W$ and $C\colon V\to G$ differential and observation operators, respectively.
Among many others, for example the classical model problem of identifying the diffusion coefficient $q$ in the linear elliptic PDE
$$ -\nabla (q \nabla u)=f \mbox{ in } \Omega$$
from measurements of $u$ in $\Omega$ can be cast in this form with $Q\subseteq L^\infty(\Omega)$, $V,W\subseteq H^1(\Omega)$, $G=L^2(\Omega)$, $A(q,u)=-\nabla (q \nabla u)$ and $C$ the embedding of $H^1(\Omega)$ into $L^2(\Omega)$.

The usual approach for tackling such inverse problems is to reduce them to an operator equation 
\be{OEq}
F(q) = g,
\ee
where $F=C\circ S$ is the composition of the parameter-to-solution map for \eqref{eq_PDE}
\be{eq_par2sol}
\begin{array}{rcl}
S\colon Q&\to&V\\
q&\mapsto&u
\end{array}
\ee
with the measurement operator $C$. 
The forward operator $F$ will then be a nonlinear operator between $Q$ and $G$ with typically unbounded inverse, so that recovery of $q$ is an ill-posed problem.
Since the given data $g^\delta$ are noisy with some noise level $\delta$
\be{eq_delta}
||g - g^\delta|| \leq \delta,
\ee
regularization is needed.

We will here as in \cite{KKVV13} consider the paradigm of the Iteratively Regularized Gauss-Newton Method (IRGNM) cf., e.g., 
\cite{Baku92,BakuKokurin,BNS94,HohageDiss,KH10,KNSBuch} and its adaptive discetization.
However, instead of reducing to \eqref{OEq}, we will simulteously consider the measurement equation and the PDE:
\begin{eqnarray}
C(u)&=&g \mbox{ in } G \label{eq_IRGNM_Cugdel0}\\ 
A(q,u)&=&f \mbox{ in } W^* \label{eq_IRGNM_Auqf0}
\end{eqnarray} 
as a system of operator equations for $(q,u)$, which we will abbreviate by
\beq\label{eq_IRGNM_Fqug}
\mathbf{F}(q,u) = \mathbf{g},
\eeq
where
\beq\label{eq_IRGNM_bfF}
  \mathbf{F}\colon Q\times V\to G\times W^*\,,\quad 
  \Ffett(q,u) = \vektor{C(u)}{A(q,u)}\,,\qquad \text{and} \qquad \mathbf{g} = \vektor{g}{f} \in G\times W^*\,.
\eeq
{ 
The noisy data for this all-at-once formulation is denoted by
\[
	\mathbf{g}^\delta = \vektor{g^\delta}{f} \in G\times W^*\,. 
\]}
This will allow us to avoid a major drawback of the method in \cite{KKVV13}, namely the necessity of solving the possibly nonlinear
PDE (to a certain precision) in each Newton step in order to evaluate $F(q)=C(S(q))$. 
Another key difference to the paper \cite{KKVV13} is that here the $u$ part of the previous iterate will not be subject to new discretization in the current iteration but keep its (usually coarser, hence cheaper) discretization from the previous step.

Therewith, we will arrive at iterations of the form 
\begin{eqnarray}
\lefteqn{(\qk, \uk)}\nonumber\\ 
&= \am_{q,u}& 
\varrho \|A'_q(\qkminuseins,\ukminuseins)(q-\qkminuseins)+A'_u(\qkminuseins,\ukminuseins)(u-\ukminuseins)
+A(\qkminuseins,\ukminuseins)-f\|_{W^*}^r
\nonumber\\
&&+\|C(\ukminuseins)+C'(\ukminuseins)(u-\ukminuseins)-g^\delta\|_G^2+ \alpha_{k} \|q-q_0\|_Q^2\,.
\label{IRGN_aao_var_intro}
\end{eqnarray}
with $\varrho>0$, $r\in\{1,2\}$.\\
For $r=2$,
this yields a least squares formulation, see Section \ref{subsec_IRGNM_ls}.\\
In case $r=1$ and $\varrho$ sufficiently large, by exactness of the norm with exponent one as a penalty, this leads to a Generalized Gauss-Newton type \cite{Bock} form of the IRGNM
\begin{eqnarray}
\lefteqn{(\qk, \uk)}\nonumber\\ 
&&= \am_{q,u} \|C(\ukminuseins)+C'(\ukminuseins)(u-\ukminuseins)-g^\delta\|_G^2+ \alpha_k \|q-q_0\|_Q^2
\nonumber\\
&&\mbox{ s.t. }A'_q(\qkminuseins,\ukminuseins)(q-\qkminuseins)+A'_u(\qkminuseins,\ukminuseins)
(u-\ukminuseins)+A(\qkminuseins,\ukminuseins)=f \mbox{
in }W^*\,.
\label{IRGN_sqp_intro}
\end{eqnarray}
see Section \ref{subsec_IRGNM_sqp}.
\begin{rem}
Although $\qk,\uk$ obviously depend on $\delta$, i.e.  $\qk = \qkdelta,\uk = \ukdelta$, 
we omit the superscript $\delta$ for better readability.
\end{rem}

All-at-once formulations have also been considered, e.g., in \cite{AscherHaber03,BirosGhattas04,BurgerMuehlhuberIP,BurgerMuehlhuberSINUM}, 
however, our approach focuses on adaptive discretization using a posteriori error estimators. Additionally it differs from the previous ones 
in the following sense: In \cite{BurgerMuehlhuberIP,BurgerMuehlhuberSINUM} a Levenberg-Marquardt approach is considered, whereas we work with 
an iterative regularized Gauss-Newton approach which allows us to also prove convegergence rates (which is an involved task in a Levenberg-Marquardt 
setting, that has been resolved only relatively recently, \cite{Hanke10}). Moreover we use a different regularization parameter choice in each Newton 
step than \cite{BurgerMuehlhuberIP,BurgerMuehlhuberSINUM}. The papers \cite{AscherHaber03,BirosGhattas04} put more emphasis on computational aspects and 
applications than we do here.

For both cases $r=1$, $r=2$ in \eqref{IRGN_aao_var_intro} we will investigate convergence and convergence rates in the continuous and adaptively discretized 
setting with discrepancy type choice of $\alpha_k$ (which in most of what follows will be replaced by $\frac{1}{\beta_k}$) and the overal stopping index $k_*$. The discretization errors with respect to certain quantities of interest will serve as refinement criteria during the Gauss-Newton iteration, where at the 
same time, we control the size of the regularization parameter. In order to estimate this discretization error we use goal-oriented error estimators (cf. \cite{BeckerRannacher,BeckerVexler}).

For the least squares case we will (for the sake of completeness but not in the main steam of this paper) also provide a result on convergence with a priori 
parameter choice in the continuous setting, see the appendix.
In Section \ref{sec_IRGNM_numericalresults}, we will provide numerical results and in Section \ref{sec_conclusions} some conclusions.

Throughout this paper, we will make the following assumptions:

\begin{ass}\label{ass_IRGNM_aao_qdagger}
There exists a solution $(q^\dagger,u^\dagger) \in  \mathcal{B}_{\rho}(q_0,u_0) \subset \mathcal{D}(A)\cap (Q\times\mathcal{D}(C))\subseteq Q\times V$ to 
\eqref{eq_IRGNM_Fqug}, where $(q_0,u_0)$ is some initial guess and $\rho$ (not to be confused with the penalty parameter $\varrho$ in \eqref{IRGN_aao_var_intro})
is the radius of the neighborhood in which local convergence of the Newton type iterations under consideration will be shown.
\end{ass}

\begin{ass}\label{ass_auinv} 
The PDE \eqref{eq_PDE} and especially also its linearization
at $(q,u)$ is uniquely and stably solvable.
\end{ass}

\begin{ass}\label{ass_IRGNM_GQnormevaluatedexactly}
The norms in $G$, $Q$, 
{  as well as the operator $C$ and the semilinear form $a\colon Q\times V\times W \to \R$ defined by the relation 
$a(q,u)(v) = \dualW{A(q,u)}{v}$ (where $\dualW{.}{.}$ denotes the duality pairing between $W^*$ and $W$)
are assumed to be evaluated exactly.}
\end{ass}

\section{A least squares formulation}\label{subsec_IRGNM_ls}

Direct application of the IRGNM to \eqref{eq_IRGNM_Cugdel0}, \eqref{eq_IRGNM_Auqf0}, i.e., to the all-at-once system \eqref{eq_IRGNM_Fqug}
yields the iteration
\begin{align}
\vektor{\qk}{\uk} 
& = \vektor{\qkminuseins}{\ukminuseins}
 - \left(\mathbf{F}'(\qkminuseins ,\ukminuseins )^* \mathbf{F}'(\qkminuseins ,\ukminuseins )
 + \left(\begin{array}{cc}\alpha_k \idQ&0\\0&\mu_k \idV\end{array}\right)\right)^{-1}\nonumber\\
& \quad \cdot \left(\mathbf{F}'(\qkminuseins ,\ukminuseins )^* (\mathbf{F}(\qkminuseins ,\ukminuseins )-\mathbf{g}^\delta) + 
\left(\begin{array}{c}\alpha_k(\qkminuseins -q_0)\\ \mu_k(\ukminuseins -u_0)\end{array}\right)\right)
\label{eq_IRGNM_aao_0}
\end{align}
with regularization parameters $\alpha_k$, $\mu_k$ for the $q$ and $u$ part of the iterates, respectively.

We will first of all show that Assumption \ref{ass_auinv} allows us to set the regularization parameter $\mu_k$ for the $u$ part to zero.
For this purpose, we introduce the abbreviations
\be{eq_IRGNM_KL} 
  K\colon V\to W^*\,,\quad K\coloneqq A'_u(q,u)\qquad\text{and}\qquad  L\colon Q\to W^*\,, \quad L\coloneqq A'_q(q,u)
\ee
with Hilbert space adjoints $K^*\colon W^*\to V$ and $L^*\colon W^*\to Q$, i.e.,  
\begin{eqnarray}\label{eq_IRGNM_Hspadj}
(Lq,w^*)_{W^*} = (q,L^*w^*)_Q \qquad 
\forall q\in Q,w^*\in W^*\,,
\nonumber\\
( Kv,w^*)_{W^*} = (v,K^*w^*)_V \qquad \forall v\in V,w^*\in W^*\,,
\end{eqnarray} 
where $(.,.)_{W^*}$ and $(.,.)_{V}$ denote the inner products in $W^*$ and $V$.

In the same way we define the Hilbert space adjoint $C'(u)^*\colon G\to V$ for $C'(u)\colon V\to G$, i.e.,
\[
(C'(u)(\delta u),\varphi)_{G} = (\delta u,C'(u)^*\varphi)_{V}\,,
\]
where $(.,.)_{G}$ denotes the inner product in $G$.

We denote the derivate of $\mathbf{F}$ at a pair $(q,u)$ by $\mathbf{T}$, i.e.,
\be{eq_IRGNM_T}
  \mathbf{T}\colon Q\times V\to G\times W^*\,,\quad 
  \mathbf{T} =\mathbf{F}'(q,u)=
 \left(\begin{array}{cc}0&C'(u)\\
 A'_q(q,u)&A'_u(q,u)\end{array}\right)
 = \left(\begin{array}{cc}0&C'(u)\\
 L &K\end{array}\right)
\ee
and define the norm
\be{eq_IRGNM_timesnormdef}
    \norm{\vektor{q}{u}}_{Q\times V}^2 \coloneqq \norm{q}_Q^2 + \norm{u}_V^2 \quad \text{and the operator norm} \quad 
    \norm{T}_{Q\times V}\coloneqq \sup_{x\in Q\times V,x\neq 0}
    \frac {\norm{Tx}_{Q\times V}}{\norm{x}_{Q\times V}}\,. 
\ee
for some $x\in Q\times V$ and some operator $T\colon Q\times V\to Q\times V$.

 Further we define 
 \be{eq_IRGNM_Ydef}
   \bY_{\alpha,\mu} \coloneqq \left(\mathbf{T}^*\mathbf{T}+
   \left(\begin{array}{cc}\alpha \idQ &0\\
   0&\mu \idV \end{array}\right)
   \right)
\ee
for $\alpha>0$, $\mu\ge 0$. 

\begin{lem}\label{lem_IRGNM_Talphamu}
Under Assumption \ref{ass_auinv}
\begin{itemize}
\item[(i)] for any $\alpha>0$, $\mu\geq0$ the inverse $\bY_{\alpha,\mu}^{-1}$ of $\bY_{\alpha,\mu}$ exists
\item[(ii)]
\[
\normQV{\bY_{\alpha,\mu}^{-1} \mathbf{T}^*\mathbf{T}}\le 1+\max\{\alpha,\mu\} \normQV{\bY_{\alpha,\mu}^{-1}}\,,
\]
\item[(iii)]
\beq\label{eq_IRGNM_Talphamu}
\normQV{\bY_{\alpha,\mu}^{-1}} 
\leq 
c_T \left(\frac 1  \alpha + 1 \right)
\eeq
for all $\alpha\in(0,1]$, $\mu\geq0$ and some $c_T>0$ independent of $\alpha,\mu$, where the bound
$c_T$ in \eqref{eq_IRGNM_Talphamu} is independent of $q$ and $u$, if the operators $K$, $K^{-1}$ and $L$, are bounded uniformly in $(q,u)$.
\end{itemize}
\end{lem}
\begin{proof}
\begin{itemize}
\item[(i):]
With the abbreviations 
\[
P=L^*L + \alpha \idQ\quad \text{and} \quad  M=C'(u)^*C'(u)+K^*K +\mu \idV
\]
we have 
\[
\bY_{\alpha,\mu} 
=\left(\begin{array}{cc} P& L^*K\\
K^*L&M \end{array}\right)\,. 
\]
{  
Since Assumption \ref{ass_IRGNM_GQnormevaluatedexactly} implies that $K$ is invertible, $M^{-1}$ exists, such that we 
can define some kind of Schur complement
\[
  N\coloneqq P-L^*K M^{-1} K^*L =L^*L+\alpha \idQ -L^*K M^{-1} K^*L\,.
\]
We will now show that $N$ is also invertible. 
Using the fact that
\begin{align*}
\|M^{-1/2}K^*\|_{W^*\to V}^2&= \|KM^{-1/2}\|_{V\to W^*}^2\\
& = \sup_{v\in V, v\not=0}\frac{\|KM^{-1/2}v\|_{W^*}^2}{\|v\|_V^2}\\
&=\sup_{v\in V, v\not=0}\frac{\|Kv\|_{W^*}^2}{\|M^{1/2}v\|_V^2}\\
&= \sup_{v\in V, v\not=0}\frac{\|Kv\|_{W^*}^2}{\|C'(u)(v)\|_G^2+\|Kv\|_{W^*}^2 +\mu \|v\|_{V}^2}\leq 1\,.
\end{align*}
for any $q\in Q$ we get 
\begin{align*}
\scalarQ{Nq}{q}
& = \scalarQ{L^*Lq +\alpha q - L^*KM^{-1}K^*Lq}{q}\\
& \ge \|Lq\|_{W^*}^2+\alpha\|q\|_Q^2- \|M^{-1/2}K^*\|_{W^*\to V}^2 \|Lq\|_{W^*}^2\\
&\geq \alpha\|q\|_Q^2\,,
\end{align*}
which implies the existence of $N^{-1}$, since $M$ and therewith also $N$ is self-adjoint. 
For
}
\be{eq_IRGNM_Tinv}
O_{\alpha\mu}\coloneqq
\left(\begin{array}{cc} N^{-1}& -N^{-1}L^*KM^{-1}\\
-M^{-1}K^*LN^{-1}&M^{-1}+M^{-1}K^*LN^{-1}L^*KM^{-1}
\end{array}\right)
\ee
there holds 
\[
O_{\alpha\mu}
\left(\begin{array}{cc} P& L^*K\\
K^*L & M
\end{array}\right) 
= \left(\begin{array}{cc} A& B\\
C & D
\end{array}\right)
\]
with 
\begin{align*}
A &\coloneqq N^{-1} \left(P  - L^*KM^{-1}K^*L\right) = \idQ\\
B &\coloneqq N^{-1}L^*K + -N^{-1}L^*KM^{-1}M = 0 \\
C &\coloneqq -M^{-1}K^*LN^{-1}P + \left(M^{-1}+M^{-1}K^*LN^{-1}L^*KM^{-1}\right)K^*L\\
  &= - M^{-1}K^*L \left[ N^{-1}\left( P - L^*KM^{-1} K^*L\right) - \idQ\right]=0\\ 
D &\coloneqq -M^{-1}K^*LN^{-1}L^*K + \left(M^{-1}+M^{-1}K^*LN^{-1}L^*KM^{-1}\right)M = \idV\,, 
\end{align*}
we have
\beq\label{eq_IRGNM_TastTinv}
O_{\alpha\mu} = \bY_{\alpha,\mu}^{-1}\,.
\eeq
\item[(ii):]  
\begin{align*}
\normQV{\bY_{\alpha,\mu}^{-1} \mathbf{T}^*\mathbf{T}}
& = \normQV{\idQV-\bY_{\alpha,\mu}^{-1} \left(\begin{array}{cc}\alpha \idQ &0\\
   0&\mu \idV \end{array}\right)}\\
&\leq 1+\max\{\alpha,\mu\} \normQV{\bY_{\alpha,\mu}^{-1}}
\end{align*}
\item[(iii):]  
{  For any $v\in V$ we get}
\begin{eqnarray*}
\scalarV{Mv}{v}
&=&\|C'(u)(v)\|_G^2+\|Kv\|_{W^*}^2+\mu \|v\|_{V}^2
\geq \|Kv\|_{W^*}^2 \geq \frac{1}{\|K^{-1}\|_{W^*\to V}^2} \|v\|_V^2\,,
\end{eqnarray*}
hence we have 
\beq\label{eq_IRGNM_normsPhatinvMinv}
\|N^{-1}\|_{Q\to Q}\leq \frac{1}{\alpha}\,, \quad 
\|M^{-1}\|_{V\to V}\leq \|K^{-1}\|_{W^*\to V}^2\,.
\eeq
For $O_{\alpha\mu}$ (cf. \eqref{eq_IRGNM_Tinv}) this yields
{
 
\begin{align*}
  \normklein{O_{\alpha\mu}}_{Q\times V}^2 
  &\le \sup_{(q,u)\in Q\times V, (q,u)\neq 0}
  \frac { \normQ{N^{-1} q + N^{-1}L^*KM^{-1}u}^2} {\normQ{q}^2 + \normV{u}^2}\\
  &\quad + \sup_{(q,u)\in Q\times V, (q,u)\neq 0}
   \frac {\normV{-M^{-1}KLN^{-1}q+\left(M^{-1}+M^{-1}K^*LN^{-1}L^*KM^{-1}\right)u}^2} {\normQ{q}^2 + \normV{u}^2}\\  
  &\le 2\left(
  \norm{N^{-1}}_{Q\to Q}^2 + \norm{N^{-1}L^*KM^{-1}}_{V\to Q}^2 \right.\\
  &\quad \left.+ \norm{M^{-1}K^*LN^{-1}}_{Q\to V}^2 +\norm{M^{-1}+M^{-1}K^*LN^{-1}L^*KM^{-1}}_{V\to V}^2\right)\\
&\le 2\norm{N^{-1}}_{Q\to Q}^2(1+ \norm{L}_{Q\to W^*}^2\norm{K}_{V\to W^*}^2\norm{M^{-1}}_{V\to V}^2)^2+ 2\norm{M^{-1}}_{V\to V}^2\\
&\le \frac{2}{\alpha^2}(1+ \norm{L}_{Q\to W^*}^2\norm{K}_{V\to W^*}^2\norm{K^{-1}}_{W^*\to V}^4)^2+2\norm{K^{-1}}_{W^*\to V}^4
\end{align*}
}

\end{itemize}
\end{proof}

Motivated by Lemma \ref{lem_IRGNM_Talphamu}, and setting $\alpha_k=\frac{1}{\beta_k}$ we define a regularized iteration by
\begin{align}
\vektor{\qk}{\uk} 
& = \vektor{\qkminuseins}{\ukminuseins}
 - \left(\mathbf{F}'(\qkminuseins ,\ukminuseins )^* \mathbf{F}'(\qkminuseins ,\ukminuseins )
 + {\frac 1 {\beta_k}} \left(\begin{array}{cc}\idQ&0\\0&0\end{array}\right)\right)^{-1}\nonumber\\
& \quad \cdot \left(\mathbf{F}'(\qkminuseins ,\ukminuseins )^* (\mathbf{F}(\qkminuseins ,\ukminuseins )
-\mathbf{g}^\delta) + {\frac 1 {\beta_k}}
\left(\begin{array}{c}\qkminuseins -q_0\\0\end{array}\right)\right)
\label{eq_IRGNM_aao}
\end{align}
or equivalently $\vektor{\qk }{\uk }$ as solution to the unconstrained minimization problem 
\begin{align}
\min_{(q,u)\in Q\times V} \cT_{\beta_k}(q,u)&
\coloneqq\|L_{k-1}(q-\qkminuseins )+K_{k-1}(u-\ukminuseins )
+A(\qkminuseins ,\ukminuseins )-f\|_{W^*}^2  \nonumber\\
&\quad +\|C(\ukminuseins )+C'(\ukminuseins )(u-\ukminuseins )-g^\delta\|_G^2+ {\frac 1 {\beta_k}} \|q-q_0\|_Q^2
\label{eq_IRGNM_aao_var}
\end{align}
with the abbreviations 
\be{eq_IRGNM_ls_abbrevKL}
L_{k-1}=A'_q(\qkminuseins ,\ukminuseins )\qquad \text{and} \qquad K_{k-1}=A'_u(\qkminuseins ,\ukminuseins )\,,
\ee
where we have set the regularization parameter for the
component $u$ to zero, which is justified by (i) in Lemma \ref{lem_IRGNM_Talphamu}.

{
 
  The optimality conditions of first order for \eqref{eq_IRGNM_aao_var} read  
  \begin{align*}
  0 &= \left(\cT_{\beta_k}\right)'_q (q,u)(\delta q)  \\
  &= 2 \scalarWstern{L_{k-1}(q -\qkminuseins )+K_{k-1}(u -\ukminuseins )+A(\qkminuseins ,\ukminuseins )-f} {L_{k-1}(\delta q)}+  2\scalarQ{q-q_0}{\delta q}\\
  0 &= \left(\cT_{\beta_k}\right)'_u (q,u)(\delta u)  \\
  &= 2 \scalarWstern{L_{k-1}(q -\qkminuseins )+K_{k-1}(u -\ukminuseins )+A(\qkminuseins ,\ukminuseins )-f} {K_{k-1}(\delta u)}\\
  &\quad+ 2 \scalarG{C(\ukminuseins )+C'(\ukminuseins )(u-\ukminuseins )-g^\delta}{C'(\ukminuseins )(\delta u)}
  \end{align*}  
}

\medskip

We refer to the Appendix for a convergence and convergence rates results for \eqref{eq_IRGNM_aao} with a priori choice of the regularization parameters and in a continuous setting.

Here we are rather interested in a posteriori parameter choice rules and adaptive discretization. 
So in each step $k$ we will replace the infinite dimensional spaces $Q,V,W$ in \eqref{eq_IRGNM_aao} by finite dimensional ones 
$Q_{h},V_{h},W_{h}=Q_{h_k},V_{h_k},W_{h_k}$
\begin{align}
\vektor{\qk}{\uk} 
=& \am_{q\in Q_h,u\in V_h}
\|L_{k-1}(q-\qold )+K_{k-1}(u-\uold )
+A(\qold ,\uold )-f\|_{W_h^*}^2  \nonumber\\
&\quad +\|C(\uold )+C'(\uold )(u-\uold )-g^\delta\|_G^2+ {\frac 1 {\beta_k}} \|q-q_0\|_Q^2\,.
\label{eq_IRGNM_aao_var_h}
\end{align}
where $(\qold,\uold)=(\qkminuseins, \ukminuseins)=(\qkminuseins_{h_{k-1}}, \ukminuseins_{h_{k-1}})$ is the previous iterate, which itself is discretized by the use of spaces $Q_{h_{k-1}},V_{h_{k-1}},W_{h_{k-1}}$.
The discretization $h_k$ may be different in each Newton step (typically it will get finer for increasing $k$), but we suppress dependence of $h$ on $k$ in our notation in most of what follows. \\
To still obtain convergence of these discretized iterates, it is essential to control the discretization error in certain quantities, which are defined, analogously to \cite{KKVV13}, via the functionals 
\begin{align*}
I_1\colon &\ Q\times V\times Q\times V \times \R\,,\to \R \\
I_2\colon &\ Q\times V\times Q\times V \,,\to \R \\
I_3 \colon &\ Q\times V\,,\to \R \\
I_4\colon &\ Q\times V\,,\to \R 
\end{align*}
where we insert the previous and current iterates $(\qold,\uold)$, $(q,u)$, respectively:
\beq \label{eq_IRGNM_aao_I1234reduced}
\begin{aligned}
I_1(\qold,\uold,q,u,\beta)
& =  \norm{\Ffett'(\qold,\uold) \binom{q-\qold}{u-\uold} 
				+\Ffett(\qold,\uold) - \gfettdelta}_{G\times W^*}^2 
				+ \frac{1}{\beta} \norm{q-q_0}_Q^2\\
& = \norm{A_q'(\qold,\uold)(q-\qold)
 			 +A_u'(\qold,\uold)(u-\uold)+A(\qold,\uold)-f}^2_{W^*}\\
& \quad + \norm{C'(u)(u-\uold)+C(\uold)-g^\delta}^2_G+\frac{1}{\beta}\norm{q-q_0}^2_Q\\
I_2(\qold,\uold,q,u) 
& = \norm{\Ffett'(\qold,\uold) \binom{q-\qold}{u-\uold} +\Ffett(\qold,\uold) - \gfettdelta}_{G\times W^*}^2\\
& = \norm{A_q'(\qold,\uold)(q-\qold)+A_u'(\qold,\uold)(u-\uold)  +A(\qold,\uold)-f}^2_{W^*}\\
& \quad+ \norm{C'(u)(u-\uold)C(\uold)-g^\delta}^2_G\\
I_3(\qold,\uold) 
& = \norm{\Ffett(\qold,\uold)-\gfettdelta}^2_{G\times W^*}\\
& = \norm{A(\qold,\uold)-f}^2_{W^*} + \norm{C(\uold)-g^\delta}^2_G\\
I_4(q,u) &= \norm{\Ffett(q,u)-\gfettdelta}^2_{G\times W^*}\\
 	 &= \norm{A(q,u)-f}^2_{W^*}+\norm{C(u)-g^\delta}^2_G\,.
\end{aligned}
\ee
and quantities of interest
%
\be{eq_IRGNM_aao_I1234k}
\begin{aligned}
I_1^k &= I_1(\qoldk,\uoldk,\qk,\uk,\beta_k) \\
I_2^k &= I_2(\qoldk,\uoldk,\qk,\uk)\\
I_3^k &= I_3(\qoldk,\uoldk) \\
I_4^k &= I_4(\qk,\uk)\,.
\end{aligned}
\ee
Their discrete analogs are correspondingly defined by
\begin{align*}
\Ieinsh \colon &\ Q\times V\times Q\times V \times \R\,,\to \R \\
\Izweih \colon &\ Q\times V\times Q\times V \,,\to \R \\
\Idreih \colon &\ Q\times V\,,\to \R \\
\Ivierh \colon &\ Q\times V\,,\to \R 
\end{align*}
\beq \label{eq_IRGNM_aao_I1234reduceddiscrete}
\begin{aligned}
\Ieinsh(\qold,\uold,q,u,\beta)
& =  \norm{\Ffett'(\qold,\uold) \binom{q-\qold}{u-\uold} 
				+\Ffett(\qold,\uold) - \gfettdelta}_{G\times W_h^*}^2 
				+ \frac{1}{\beta} \norm{q-q_0}^2_Q\\
& = \norm{A_q'(\qold,\uold)(q-\qold)
 			 +A_u'(\qold,\uold)(u-\uold)+A(\qold,\uold)-f}^2_{W_h^*}\\
& \quad + \norm{C'(u)(u-\uold)+C(\uold)-g^\delta}^2_G+\frac{1}{\beta}\norm{q-q_0}^2_Q\\
\Izweih(\qold,\uold,q,u) 
& = \norm{\Ffett'(\qold,\uold) \binom{q-\qold}{u-\uold} +\Ffett(\qold,\uold) - \gfettdelta}_{G\times W_h^*}^2\\
& = \norm{A_q'(\qold,\uold)(q-\qold)+A_u'(\qold,\uold)(u-\uold)  +A(\qold,\uold)-f}^2_{W_h^*}\\
& \quad+ \norm{C'(u)(u-\uold)C(\uold)-g^\delta}^2_G\\
\Idreih(\qold,\uold) 
& = \norm{\Ffett(\qold,\uold)-\gfettdelta}^2_{G\times W_h^*}\\
& = \norm{A(\qold,\uold)-f}^2_{W_h^*} + \norm{C(\uold)-g^\delta}^2_G\\
\Ivierh(q,u) &= \norm{\Ffett(q,u)-\gfettdelta}^2_{G\times W_h^*}\\
 	 &= \norm{A(q,u)-f}^2_{W_h^*}+\norm{C(u)-g^\delta}^2_G
\end{aligned}
\ee
and
\be{eq_IRGNM_aao_I1234kdiscrete}
\begin{aligned}
I_{1,h}^k &= I_{1,h}(\qoldk,\uoldk,\qhk,\uhk,\beta_k)\\
I_{2,h}^k &= I_{2,h}(\qoldk,\uoldk,\qhk,\uhk,) \\
I_{3,h}^k &= I_{3,h}(\qoldk,\uoldk) \\
I_{4,h}^k &= I_{4,h}(\qhk,\uhk)\,.
\end{aligned}
\ee 
At the end of each iteration step we set 
\be{eq_IRGNM_ggn_qoldkplus}
\qoldkplus = \qhk \quad \text{and}\quad  \uoldkplus = \uhk\,.
\ee
\begin{rem}
Note that here neither $\qold$ nor $\uold$ are subject to new adaptive discretization in the current step, but
they are taken as fixed quantities from the previous step. This is different from \cite{KKVV13},
where $\uold$ also depends on the current discretization.

For \eqref{eq_IRGNM_aao_I1234k} and \eqref{eq_IRGNM_aao_I1234kdiscrete} we assume that the  
norms in $G$ and $Q$ are evaluated exactly cf. Assumption \ref{ass_IRGNM_GQnormevaluatedexactly}.


In our convergence proofs we will compare the quantities of interest $I_{i,h}^{k}$ with those $I_i^{k}$ that would be obtained with exact computation on the infinite dimensional spaces, starting from the same $(\qold,\uold)=({\qold}_{h_{k-1}},{\uold}_{h_{k-1}})$ as the one underlying $I_{i,h}^{k}$. Thus, in our analysis besides the actually computed sequence $(\qk_h,\uk_h)=(\qk_{h_k},\uk_{h_k})$ there appears an auxiliary sequence $(\qk,\uk)$, see Figure \ref{fig:iteration}.

\begin{figure}
\begin{center}
\setlength{\unitlength}{2cm}
\begin{picture}(5,4)
\thinlines
\put(0,0.5){\vector(1,0){4.5}}
\put(-0.05,0.45){$\bullet$}
\put(0.95,0.45){$\bullet$}
\put(1.95,0.45){$\bullet$}
\put(2.95,0.45){$\bullet$}
\put(3.95,0.45){$\bullet$}
\put(-0.2,0.3){$k=0$}
\put(0.8,0.3){$k=1$}
\put(1.8,0.3){$k=2$}
\put(2.8,0.3){$k=3$}
\put(3.8,0.3){$k=4$}
\thicklines
\put(0,1){\line(2,1){1}}
\put(1,1.5){\line(4,3){1}}
\put(2,2.25){\line(1,1){1}}
\put(3,3.25){\line(4,3){1}}
\put(-0.05,0.95){$\bullet$}
\put(0.95,1.45){$\bullet$}
\put(1.95,2.2){$\bullet$}
\put(2.95,3.2){$\bullet$}
\put(3.95,3.95){$\bullet$}
\put(0.05,0.85){$(q_0,u_0)$}
\put(1.05,1.35){$(q^1_{h_1},u^1_{h_1})=(\qoldzwei,\uold^2)$}
\put(2.05,2.1){$(q^2_{h_2},u^2_{h_2})=(\qolddrei,\uold^3)$}
\put(3.05,3.1){$(q^3_{h_3},u^3_{h_3})=(\qoldvier,\uold^4)$}
\put(4.05,3.85){$(q^4_{h_4},u^4_{h_4})$}
\thinlines
\put(0,1){\line(3,2){1}}
\put(1,1.5){\line(2,1){1}}
\put(2,2.25){\line(2,1){1}}
\put(3,3.25){\line(2,1){1}}
\put(0.95,1.61){$\diamond$}
\put(1.95,1.95){$\diamond$}
\put(2.95,2.7){$\diamond$}
\put(3.95,3.7){$\diamond$}
\put(0.35,1.75){$(q_1,u_1)$}
\put(2.0,1.8){$(q_2,u_2)$}
\put(3.05,2.6){$(q_3,u_3)$}
\put(4.05,3.6){$(q_4,u_4)$}

\end{picture}
\end{center}
\caption{Sequence of discretized iterates and auxiliary sequence of continuous iterates for the 
all-at-once formulation of IRGNM}\label{fig:iteration}
\end{figure}
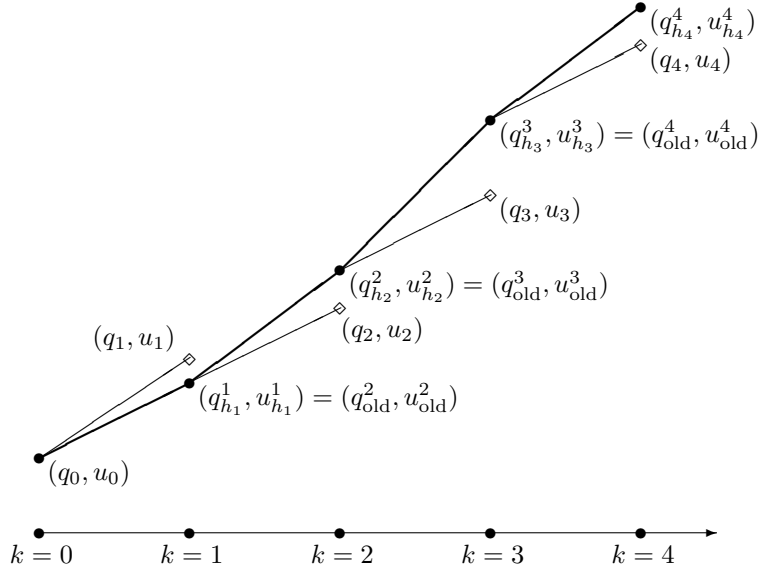

%
%

\end{rem}

We assume the knowledge about bounds $\eta_i^k$ on the error in the quantities of interest due to discretization  
\beq\label{eq_IRGNM_intcond}
|I_{i,h}^{k} - I_i^{k}| \leq \eta_i^{k} \,, \quad i\in\{1,2,3,4\}
\eeq
{ (, which can, at least partly, be computed by goal oriented error estimators, see e.g., \cite{BeckerRannacher,BeckerVexler,GKV,KKV10} 
and Section\ref{subsec_estimators_ls})} 
and to refine adaptively according to these bounds. On the other hand, we will now impose conditions on such upper bounds for the discretization error that enable to prove convergence and convergence rates results, see Assumption \ref{ass_IRGNM_etacond} below.

Additionally, we will make some assumptions on the forward operator
\begin{ass}\label{ass_IRGNM_aao_Fweakseqclosed}
Let the reduced forward operator $\bF$ be continuous and weakly sequentially closed, i.e.
\begin{eqnarray}
&&(q_n \rightharpoonup q \land u_n \rightharpoonup u \land C(u_n)\to g
\land A(q_n,u_n) \to f) \nonumber\\
&&\Rightarrow (u\in \mathcal{D}(C) \land (q,u) \in \mathcal{D}(A) \land C(u)=g  \land A(q,u) = f)
\nonumber
\end{eqnarray}
for all sequences $\left((q_n,u_n)\right)_{n \in \N} \subseteq Q\times V$.
\end{ass}
We also transfer the usual tangential cone condition to the all-at once setting from this section, which yields 
\begin{ass}\label{ass_IRGNM_aao_tangcone}
Let 
 \begin{align*}
 &\|C(u) - C(\bar{u})  - C'(u)(u - \bar{u})\|_{G}
    +\|A(q,u) - A(\bar{q},\bar{u}) - A'_q(q,u)(q - \bar{q}) - A'_u(q,u)(u - \bar{u})\|_{W^*}\nonumber\\
 &\leq c_{tc} \left(\| C(u) - C(\bar{u})\|_{G} +\| A(q,u) - A(\bar{q},\bar{u})\|_{W^*}\right)\nonumber
 \end{align*}
hold for all $(q,u), (\bar{q},\bar{u}) \in \mathcal{B}_{\rho}(q_0,u_0) \subset (Q\times V)$ and some $0 < c_{tc} <
1$. 
\end{ass}

The choice of the regularization parameter $\beta_k$ will be done a posteriori according to an inexact Newton /discrepancy principle, which with the quantities introduced above reads as  
\beq\label{eq_IRGNM_inexNewton}
{\tilde{\underline\theta}} I_{3,h}^{k} \leq 
I_{2,h}^{k}
\leq {\tilde{\overline\theta}} I_{3,h}^{k} \,.
\eeq
A discrepancy type principle will also be used for the choice of the overall stopping index
\beq\label{eq_IRGNM_stop}
k_* = \min \{ k\in \mathbb{N} ~ : ~ I_{3,h}^{k} \leq \tau^2\delta^2 \}.
\eeq 

The parameters used there have to satisfy the following assumption.
\begin{ass}\label{ass_IRGNM_tauthetacond}
Let $\tau$ and $\tilde{\underline\theta}$ be chosen sufficiently large and $\tilde{\overline\theta}$ sufficiently small 
(see \eqref{eq_IRGNM_inexNewton},\eqref{eq_IRGNM_stop}), such that
\beq\label{eq_IRGNM_taucond}
2\left(c_{tc}^2 + \frac{(1+c_{tc})^2}{\tau^2}\right) < \tilde{\underline\theta}
\quad \textrm{and}\quad 
\frac{2\tilde{\overline\theta} + 4c_{tc}^2}{1-4c_{tc}^2} < 1\,.
\eeq
\end{ass}

Therewith, we can also formulate our conditions on precision in the quantities of interest:
\begin{ass}\label{ass_IRGNM_etacond}
Let for the discretization error with respect to the quantities of interest estimate \eqref{eq_IRGNM_intcond} hold,
where $\eta_1^k$, $\eta_2^k$, $\eta_3^k$, $\eta_4^k$ are selected such that
\beq\label{eq_IRGNM_etacond1}
\eta_1^{k} + 4 c_{tc}^2 \eta_3^{k}\leq
\left({\tilde{\underline\theta}} - 2\left(2c_{tc}^2 + \frac{(1+2c_{tc})^2}{\tau^2}\right)\right)
I_{3,h}^{k}
\eeq
\beq\label{eq_IRGNM_etacond2}
\eta_3^{k}\leq c_1 I_{3,h}^{k} \ \mbox{ and } \
\eta_2^{k}\to0\,, \ \eta_3^{k}\to0\,, \ \eta_4^{k}\to0 \mbox{ as }k\to\infty 
\eeq
\beq\label{eq_IRGNM_etacond3}
I_{3,h}^{k}\leq (1+c_3)I_{4,h}^{k-1}+r^k \ \mbox{ and } \
(1+c_3)\frac{2{\tilde{\overline\theta}} + 4 c_{tc}^2}{1-4c_{tc}^2}\leq c_2<1
\eeq
for some constants $c_1, c_2, c_3 >0$, and a sequence $r^k\to0$ as $k\to\infty$ (where the second condition in \eqref{eq_IRGNM_etacond3} 
is possible due to the right inequality in \eqref{eq_IRGNM_taucond}).
\end{ass}

Exactly along the lines of the proofs of Theorems 1 and 2 in in \cite{KKVV13}, replacing $F$ there by $\mathbf{F}$ according to \eqref{eq_IRGNM_bfF}, we therewith obtain convergence and convergence
rates results:
\begin{thm}\label{thm_IRGNM_conv_aao}
Let the Assumptions \ref{ass_IRGNM_aao_qdagger}, \ref{ass_auinv}, \ref{ass_IRGNM_GQnormevaluatedexactly}, \ref{ass_IRGNM_aao_Fweakseqclosed} and \ref{ass_IRGNM_aao_tangcone}
with $c_{tc}$ sufficiently small be satisfied and let Assumption \ref{ass_IRGNM_tauthetacond} hold.
For the quantities of interest \eqref{eq_IRGNM_aao_I1234k} and \eqref{eq_IRGNM_aao_I1234kdiscrete},
let, further, the estimate \eqref{eq_IRGNM_intcond} hold with $\eta_i$ satisfying Assumption \ref{ass_IRGNM_etacond}. 

Then with $\beta_k$, $h=h_k$ fulfilling \eqref{eq_IRGNM_inexNewton}, $k_*$ selected according to \eqref{eq_IRGNM_stop}, and 
$(\qhkk,\uhkk)$ defined by \eqref{eq_IRGNM_aao_var_h} there holds
 \begin{itemize}
 \item[(i)]
 \beq\label{eq_IRGNM_ineq0_ls}
  \|\qhkk -q_0\|_Q^2+\|\uhkk -u_0\|_V^2 \leq \|q^\dag -q_0\|_Q^2+\|u^\dag -u_0\|_V^2 \qquad \forall 0\le k\leq k_*\,;
 \eeq
 \item[(ii)] $k_*$ is finite\,;
 \item[(iii)] $(\qhkstern,\uhkstern)=(q_{h_{k_*(\delta)}}^{k_*(\delta),\delta},u_{h_{k_*(\delta)}}^{k_*(\delta),\delta})$ converges (weakly) subsequentially to a solution of \eqref{eq_IRGNM_Fqug} as 
 $\delta \to 0$
in the sense that it has a weakly convergent subsequence and each
weakly convergent subsequence converges strongly to a solution of \eqref{eq_IRGNM_Fqug}.
If the solution $(q^\dagger,u^\dagger)$ to \eqref{eq_IRGNM_Fqug} is unique, then $(\qhkstern,\uhkstern)$ 
converges strongly to $(q^\dagger,u^\dagger)$  as $\delta\to0$.
\end{itemize}
\end{thm}

For proving rates, as usual (cf. e.g. \cite{BakuKokurin,EHNBuch,HohageDiss,KNSBuch}) source conditions are assumed
\begin{ass}\label{ass_IRGNM_source_aao}
Let 
\[
  (q^\dag -q_0,u^\dag-u_0)
  \in \cR \left(\kappa\left(\bF'(q^\dag,u^\dag)^*\bF'(q^\dag,u^\dag)\right)\right)
\]
(cf. Assumption \ref{ass_IRGNM_aao_qdagger})
hold with some $\kappa\colon \R^+\to\R^+$ such that $\kappa^2$ is strictly monotonically increasing on
$(0, \|\bF(q^\dag,u^\dag)\|_{Q\times V}^2]$,
$\phi$ defined by $\phi^{-1}(\lambda)=\kappa^2(\lambda)$ is convex and $\psi$ defined by
  $\psi(\lambda)=\kappa(\lambda)\sqrt{\lambda}$ is  strictly monotonically increasing on
  $(0,\|\bF(q^\dag,u^\dag)\|_{Q\times V}^2]$. 
Here, for some selfadjoint nonnegative operator $A$, the operator function $\kappa(A)$ is defined via functional calculus based on the spectral theorem (cf. e.g. \cite{EHNBuch}).
\end{ass}

\begin{thm} \label{thm_IRGNM_rates_aao}
Let the conditions of Theorem \ref{thm_IRGNM_conv_aao} and additionally
the source condition Assumption \ref{ass_IRGNM_source_aao} 
be fulfiled. 

Then there exists a $\bar{\delta}>0$ and a constant $\bar{C}>0$ independent of $\delta$ such that
for all $\delta\in (0,\bar{\delta}]$
the convergence rates 
 \begin{equation}\label{eq_IRGNM_rates0_aao}
   \normQklein{\qold^{k_*}-\qdag}^2+ \normVklein{\uold^{k_*}-u^\dag}^2
 = \mathcal{O}\left(\frac{ \delta^2 }{\psi^{-1}(\bar{C}\delta)}\right)\,.
 \end{equation}
 are obtained.
\end{thm}

\begin{rem}
We compare the source conditions for the reduced formulation
\begin{equation}\label{eq_nonlineartikhonov_sourceconditiongeneral}
   \qdag-q_0 \in \mathcal{R}\left(\kappa\left(F'(\qdag)^*F'(\qdag)\right)\right)
\end{equation}
with Assumption \ref{ass_IRGNM_source_aao} for the
all-at-once formulation, e.g. in the case $\kappa(\lambda)=\sqrt{\lambda}$. 
Namely, in that case
\eqref{eq_nonlineartikhonov_sourceconditiongeneral} reads: There exists $\ol g\in G$ such that 
\beq\label{eq_IRGNM_source12}
q^\dag-q_0=F'(q^\dag)^*\ol g= S'(q^\dag)^*C'(S(q^\dag))^*\ol g\,. 
\eeq
On the other hand, 
Assumption \ref{ass_IRGNM_source_aao} with the same $\kappa$ reads: There exists 
$\tilde{\mathbf{g}}=\vektor{\tilde g}{\tilde f} \in G \times W^*$ such that 
\[
  (q^\dag -q_0,u^\dag-u_0)
= \bF'(q^\dag,u^\dag)^*\tilde{\mathbf{g}} 
  = \left(\begin{array}{cc}0&A'_q(q^\dag,u^\dag)^*\\
  C'(u^\dag)^*&A'_u(q^\dag,u^\dag)^*\end{array}\right)\tilde{\mathbf{g}}\,,
\]
which is equivalent to 
\begin{eqnarray*}
q^\dag-q_0&=&A'_q(q^\dag,u^\dag)^*\tilde f\\
u^\dag-u_0&=&C'(u^\dag)^*\tilde g+A'_u(q^\dag,u^\dag)^*\tilde f\,,
\end{eqnarray*}
and by elimination of $\tilde f$ and use of the identities $u^\dag = S(q^\dag)$ and 
$S'(q^\dag)= -A_u'(q^\dag,u^\dag)^{-1} A'_q(q^\dag,u^\dag)$ we get
\[
  q^\dag -q_0 = S'(q^\dag)^* \Bigl(C'(S(q^\dag))^*\tilde g + u_0- u^\dag\Bigr)\,,
\]
which, setting $\overline g = \tilde g + C'(u^\dag)^{-*} (u_0- u^\dag)$ becomes \eqref{eq_IRGNM_source12}, provided $u_0- u^\dag\in \cR(C'(u^\dag)^*)$.
\end{rem}

\subsection{Computation of the error estimators}\label{subsec_estimators_ls}

Theoretically the error estimators for this subsection can be computed similarly to those from 
\cite{KKVV13}. The fact that we consider an unconstrained optimization problem should make things
easier, but we get another problem in return: For estimating $I_1$ and $I_2$ we would have to estimate terms
like
\[
  \norm{E(q,u)}_{W^*} - \norm{E(\qh,\uh)}_{W_h^*}
\]
for some operator $E \colon Q\times V \to W^*$, which would be quite an effort to do via goal oriented error estimators. 
For this reason, the presented least squares formulation will not be implemented
and we will not go into more detail concerning the error estimators for this section.

\section{A Generalized Gauss-Newton formulation} \label{subsec_IRGNM_sqp}

A drawback of the unconstrained formulation \eqref{eq_IRGNM_aao_var} is the necessity of computing the $W^*$-norm
of the (linearized) residual and especially of computing error estimators for this quantity of interest.
Besides, a rescaling of the state equation \eqref{eq_IRGNM_Auqf0} changes the solution of the optimization problem.
Moreover, depending on the given inverse problem and its application, in some cases, it does not make sense to only minimize
the residual of the linearized state equation, instead of setting it to zero.  

A formulation that is much better tractable is obtained by defining $(\qk, \uk)=(\qkdelta, \ukdelta)$ as a solution to the PDE
constrained minimization problem
\begin{align}
\min_{(q,u)\in Q\times V} &\mathcal{T}_{\beta_k}(q,u)\coloneqq\|C(\ukminuseins)+C'(\ukminuseins)(u-\ukminuseins)-g^\delta\|_G^2\nonumber\\
&\qquad\qquad\qquad+ {\frac 1 {\beta_k}} \left(\|q-q_0\|_Q^2+\|u-u_0\|_V^2\right)\label{eq_IRGNM_IRGNMsqp}\\
&\text{s.t.} \quad 
L_{k-1}(q-\qkminuseins)+K_{k-1}(u-\ukminuseins)
+A(\qkminuseins,\ukminuseins)=f \qquad \text{in } W^* \label{eq_IRGNM_ggn_state}
\end{align}
(see also \cite{BurgerMuehlhuberIP}, \cite{BurgerMuehlhuberSINUM}) 
with the abbreviations \eqref{eq_IRGNM_ls_abbrevKL}.

We consider the Lagrangian $\L\colon Q\times V\times W\to \R$ 
\[
  \L(q,u,z) \coloneqq \cT_{\beta_k}(q,u) + \langle f - A(\qkminuseins,\ukminuseins) - L_{k-1}(q-\qkminuseins)- K_{k-1}(u-\ukminuseins) ,z\rangle_{W^*,W}
\]
and formulate the optimality conditions of first order for \eqref{eq_IRGNM_IRGNMsqp}:
\begin{align}
   \L_z'(q,u,z)(\delta z)  
     &=  \langle f - A(\qkminuseins,\ukminuseins) - L_{k-1}(q-\qkminuseins)- K_{k-1}(u-\ukminuseins) ,\delta z\rangle_{W^*,W}
     =0\,,\label{eq_ggn_Lz}\\        
   \L_u'(q,u,z)(\delta u)
     &= 2 \scalarG{C(\ukminuseins)+C'(\ukminuseins)(u-\ukminuseins)-g^\delta}{C'(\ukminuseins)(\delta u)} \nonumber\\
     &\quad + \tfrac 2 {\beta_k} \scalarV{u-u_0}{\delta u}               
      - \dualW{ K_{k-1} \delta u}{z}=0\,,\label{eq_ggn_Lu}\\
   \L_q'(q,u,z)(\delta q)
  	&= \tfrac 2 {\beta_k} \scalarQ{q-q_0}{\delta q}- \dualW{L_{k-1} \delta q}{z}=0\label{eq_ggn_Lq}
\end{align}
for all $\delta q\in Q$, $\delta u \in V$, $\delta z \in W$.

We assume boundedness of the operators  $A(q,u),L_{k-1},K_{k-1}^*,K_{k-1}^{-1},C(u)$ and $C'(u)$
in the following sense.
\begin{ass}\label{ass_IRGNM_unifbyA}
There holds 
\[
\sup_{(q,u)\in\mathcal{B}_{\rho}(q_0,u_0)}
\|A(q,u)\|_{W^*}
+ \|A'_q(q,u)\|_{Q\to W^*} 
+ \|A'_u(q,u)^*\|_{W^*\to V} 
 \|A'_u(q,u)^{-1}\|_{W^*\to V} <\infty
\]
and 
\[
\sup_{u\in\mathcal{B}_{\rho}(u_0)}
\{\|C(u)\|_{G}
+ \|C'(u)\|_{V\to G} \}<\infty\,.
\]
\end{ass}

The following lemma about boundedness of the adjoint variable will serve as tool for unformly bounding the penalty parameter $\varrho$.
\begin{lem}\label{lem_IRGNM_adjointbounded}
Under Assumption \ref{ass_IRGNM_unifbyA} and provided $(\qkminuseins,\ukminuseins)\in \mathcal{B}_{\rho}(q_0,u_0)$, 
{ 
for a stationary point $(\qk,\uk,\zk) \in Q \times  V \times W$ of $\L$ (cf. \eqref{eq_ggn_Lz} - \eqref{eq_ggn_Lq})
there holds the estimate 
\be{eq_IRGNM_helpadjointresult}
  \|\zk\|_{W}\leq \cadj \left(\|\qkminuseins -q_0\|_Q+\|\ukminuseins-u_0\|_V+1\right)\,,
\ee
with a constant $\cadj$ independent of $k$.
}
\end{lem}
\begin{proof}
To formulate the optimality system \eqref{eq_ggn_Lz}-\eqref{eq_ggn_Lq} in a matrix-vector form, we introduce another dual variable
$p \in W^*$ defined by 
\be{eq_IRGNM_pdef}
  p =J_{W^*} z \in W^*
\ee
via the map $J_{W^*}$, which maps $z\in W$ to the Riesz representation
$J_{W^*}z\in W^*$ of the linear functional $W^*\to \R$, $w^*\mapsto w^*(p)$, such that
\[
  \L(q,u,z) = \cT_{\beta_k}(q,u) + (f - A(\qkminuseins,\ukminuseins) - L_{k-1}(q-\qkminuseins)- K_{k-1}(u-\ukminuseins,) ,p)_{W^*}\,.
\]
Using the abbreviations \eqref{eq_IRGNM_ls_abbrevKL} and 
\be{eq_IRGNM_abbrevCr}
\Ck\coloneqq C'(\ukminuseins) \,, \quad \rf\coloneqq A(\qkminuseins,\ukminuseins)-f\,, \quad \rg \coloneqq C(\ukminuseins)-g^\delta
\ee 
the optimality system \eqref{eq_ggn_Lz}-\eqref{eq_ggn_Lq} can be written as
\begin{align*}
   \L_z'(q,u,z)(\delta z)  
     &=  \langle f - A(\qkminuseins,\ukminuseins) - L_{k-1}(q-\qkminuseins)- K_{k-1}(u-\ukminuseins) ,\delta z\rangle_{W^*,W}=0\,,\\        
   \L_u'(q,u,z)(\delta u)
     &= 2\scalarG{\rg+\Ck(u-\ukminuseins)}{\Ck \delta u} + \tfrac 2 {\beta_k} \scalarV{u-u_0}{\delta u}- \scalarWstern{K \delta u}{p}\\
     & = \scalarV{2\Ck^*[\rg+\Ck(u-\ukminuseins)] + \tfrac 2 {\beta_k} (u-u_0) - K_{k-1}^*p}{\delta u}  =0\,,\\
   \L_q'(q,u,z)(\delta q)
  	&= \tfrac 2 {\beta_k} \scalarQ{q-q_0}{\delta q} -\scalarWstern{L_{k-1}\delta q}{p}
  	= \scalarQ{\tfrac 2 {\beta_k} (q-q_0)-L_{k-1}^*p}{\delta q}=0
\end{align*}
for all $\delta q\in Q$, $\delta u \in V$ and $\delta z\in W$, or equivalently as 
\begin{align*}
  \qk &= q_0 + \frac {\beta_k} 2 \Lk^*\pk \\
   \uk &= \left[\frac 2 {\beta_k}\idV + 2\Ck^*\Ck\right]^{-1} 
   \left(2\Ck^*\left(\Ck(\ukminuseins)-\rg\right)+\frac 2 {\beta_k}u_0 + \Kk^*\pk \right)\\
   \uk &=  \Kk^{-1} \left(\Lk\qkminuseins +\Kk \ukminuseins - \rf -  \Lk\qk \right)\,.
\end{align*}
Eliminating $\qk$ and $\uk$ this yields 
\begin{align*}
   &\left[\frac 2 {\beta_k}\idV + 2\Ck^*\Ck\right]^{-1} 
   \left(2\Ck^*\left(\Ck(\ukminuseins)-\rg\right)+\frac 2 {\beta_k}u_0 + \Kk^* \pk \right)\\
    &=  \Kk^{-1} \left(\Lk\qkminuseins +\Kk \ukminuseins - \rf -  
    \Lk  \left(q_0 + \frac {\beta_k} 2 \Lk^*\pk\right) \right)\,,
\end{align*}
which we reformulate as
\begin{align*}
 &-\frac {\beta_k} 2\Kk^{-1} \Lk  \Lk^*\pk
 - \left[\frac 2 {\beta_k}\idV + 2\Ck^*\Ck\right]^{-1}\Kk^* \pk\\
 &=   \left[\frac 2 {\beta_k}\idV + 2\Ck^*\Ck\right]^{-1} 
   \left(2\Ck^*\left(\Ck(\ukminuseins)-\rg\right)+\frac 2 {\beta_k}u_0\right)\\ 
   & \quad -  \Kk^{-1} \left(\Lk(\qkminuseins-q_0) +\Kk \ukminuseins- \rf\right)
\end{align*}
and finally
  \begin{align*}
 &-\left[\frac 1 {\beta_k}\idV + \Ck^*\Ck\right] \beta_k \Kk^{-1} \Lk  \Lk^*\pk
 - \Kk^* \pk\\
 &= 2\Ck^*\left(\Ck(\ukminuseins)-\rg\right)+\frac 2 {\beta_k}u_0\\ 
   & \quad - 2\left[\frac 1 {\beta_k}\idV + \Ck^*\Ck\right]\Kk^{-1} \left(\Lk(\qkminuseins-q_0) +\Kk \ukminuseins- \rf\right)\,.
\end{align*}
With 
\[
C_\beta\coloneqq\left({\frac 1 {\beta_k}} \idV+\Ck^*\Ck\right)^{1/2}
\]
this is equivalent to
 \begin{align*}
 & -\beta_k  C_\beta^2 \Kk^{-1} \Lk  \Lk^*\pk
 - \Kk^* \pk\\
 &= 2\Ck^*\left(\Ck(\ukminuseins)-\rg\right)+\frac 2 {\beta_k}u_0\\ 
   & \quad - 2 C_\beta^2\Kk^{-1} \left(\Lk(\qkminuseins-q_0) +\Kk \ukminuseins- \rf\right)\,,
\end{align*}
which upon premultiplication with $C_\beta^{-1}$ becomes
 \begin{align*}
  & -\left(\beta_k C_\beta \Kk^{-1} \Lk  (\Kk^{-1} \Lk)^* C_\beta+ \idV\right)C_\beta^{-1} \Kk^* \pk\\
  &= -2C_\beta^{-1}\Ck^*\rg+\frac 2 {\beta_k}C_\beta^{-1}(u_0-\ukminuseins) 
  + 2 C_\beta \Kk^{-1} \left(\Lk(q_0-\qkminuseins) + \rf\right) \\
    & \quad + \left[ 2 C_\beta^{-1} \Ck^* \Ck- 2 C_\beta 
    +\frac 2 {\beta_k} C_\beta^{-1}\right]\ukminuseins\\
  & = -2C_\beta^{-1}\Ck^*\rg+\frac 2 {\beta_k}C_\beta^{-1}(u_0-\ukminuseins) 
  + 2 C_\beta \Kk^{-1} \left(\Lk(q_0-\qkminuseins) + \rf\right) \\
    & \quad + \left[ 2 C_\beta^{-1} \left(\frac 1 {\beta_k} \idV+ \Ck^* \Ck\right)- 2 C_\beta \right]\ukminuseins   \\
  & = -2C_\beta^{-1}\Ck^*\rg+\frac 2 {\beta_k}C_\beta^{-1}(u_0-\ukminuseins) 
  + 2 C_\beta \Kk^{-1} \left(\Lk(q_0-\qkminuseins) + \rf\right) \,.  
 \end{align*}
 
Since $\beta_k C_\beta \Kk^{-1} \Lk  (\Kk^{-1} \Lk)^* C_\beta = 
\beta_k (C_\beta \Kk^{-1} \Lk) (C_\beta\Kk^{-1} \Lk)^*$  
is positive semidefinite, we can conclude
\begin{eqnarray*}
  \lefteqn{\norm{C_\beta^{-1} \Kk^* \pk}_V}\\
  &\le& 
  \norm{-2C_\beta^{-1}\Ck^*\rg+\frac 2 {\beta_k}C_\beta^{-1}(u_0-\ukminuseins) 
  + 2 C_\beta \Kk^{-1} \left(\Lk(q_0-\qkminuseins) + \rf\right)}_V  \,,
\end{eqnarray*}
and with the estimates
\[
 \|C_\beta^{-1} \Ck^*\|_{G\to V}\leq 1 \,, \quad \|C_\beta^{-1}\|_V\leq \beta^{\frac 12}\,, \quad 
 \quad \text{and} \quad \|C_\beta\|_V \leq \left(\frac 1 {\beta_k} + \norm{\Ck}^2_{V\to G}\right)^{\frac 12}
 \]
we have
 \begin{align*}
 \norm{\Kk^* \pk}_V
 \le& \norm{C_\beta C_\beta^{-1}\Kk^* \pk}_V\\
 \le& \left( \frac 1 {\beta_k} + \norm{\Ck}_{V\to G}^2\right)^{\frac 12}
 \norm{C_\beta^{-1}\Kk^* \pk}_V\\
 \le& 2 \left(\frac 1 {\beta_k}  +\|\Ck\|_{V\to G}^2 \right)^{\frac 12}
    \Bigl\{
      \|\rg\|+ \frac 1 {\sqrt{\beta_k}}\|u_0-\ukminuseins\|_V    \\
  &\qquad\qquad  + \left(\frac 1 {\beta_k}+\|\Ck\|_{V\to G}^2\right)^{\frac 12 }
     \|\Kk^{-1}\left(\Lk(q_0-\qkminuseins)+\rf\|\right)
    \Bigr\}\,,
 \end{align*}
which by Assumption \ref{ass_IRGNM_unifbyA} and $(\qkminuseins,\ukminuseins)\in \mathcal{B}_{\rho}(q_0,u_0)$ yields \eqref{eq_IRGNM_helpadjointresult}.
%
\end{proof}

We will prove inductively that the iterates indeed remain in $\mathcal{B}_{\rho}(q_0,u_0)$, see estimate \eqref{eq_IRGNM_ineq0_sqp} below. Thus, due to Lemma \ref{lem_IRGNM_adjointbounded}, 
	 which remains valid in the discretized setting \eqref{eq_IRGNM_IRGNMsqpdiscrete}, we get uniform boundedness of the dual variables by some sufficiently large $\varrho$, namely 
\begin{equation}\label{eq_IRGNM_rhobar}
\varrho\geq \cadj \left(\|q^\dag -q_0\|_Q+\|u^\dag-u_0\|_V + 1\right)\,.
\end{equation}
Hence we can use exactness of the norm with exponent one as a penalty (cf., e.g., Theorem 5.11 in \cite{GeigerKanzow}), which implies that a solution $(\qk,\uk)$ of \eqref{eq_IRGNM_IRGNMsqp}, \eqref{eq_IRGNM_ggn_state} coincides with the unique solution of the unconstrained minimization problem
\be{eq_IRGNM_aao_var_r}
\begin{aligned}
\min_{(q,u)\in Q\times V} &\varrho \|A'_q(\qkminuseins ,\ukminuseins )(q-\qkminuseins )
+A'_u(\qkminuseins ,\ukminuseins )(u-\ukminuseins )
+A(\qkminuseins ,\ukminuseins )-f\|_{W^*}\\
&\quad +\|C(\ukminuseins )+C'(\ukminuseins )(u-\ukminuseins )-g^\delta\|_G^2
+ {\frac 1 {\beta_k}} (\|q-q_0\|_Q^2+\|u-u_0\|_V^2)\,,
\end{aligned}
\ee
for $\rho$ larger than the norm of the dual variable.
The formulation \eqref{eq_IRGNM_aao_var_r} of \eqref{eq_IRGNM_IRGNMsqp}, \eqref{eq_IRGNM_ggn_state} will be used in the convergence proofs only. For a practical implementation we will directly discretize \eqref{eq_IRGNM_IRGNMsqp}, \eqref{eq_IRGNM_ggn_state}.

The discrete version of \eqref{eq_IRGNM_IRGNMsqp}, \eqref{eq_IRGNM_ggn_state} reads 
\be{eq_IRGNM_IRGNMsqpdiscrete}
\begin{aligned}
\min_{(q,u)\in Q_{h_k}\times V_{h_k}} &
\|C(\uold)+C'(\uold)(u-\uold)-g^\delta\|_G^2\\
& + {\frac 1 {\beta_k}} \left(\|q-q_0\|_Q^2+\|u-u_0\|_{V_{h_k}}^2\right)
\end{aligned}
\ee
\be{eq_IRGNM_ggn_statediscrete}
\text{s.t.} \quad L_{k-1}(q-\qold)+K_{k-1}(u-\uold)
+A(\qold,\uold)=f \qquad \text{in } W_{h_k}^*\,,
\ee
where $(\qold,\uold)=(\qkminuseins, \ukminuseins)=(\qkminuseins_{h_{k-1}}, \ukminuseins_{h_{k-1}})$ is the previous iterate
and we assume again that the norms in $G$ and $W$ as well as $A$ and $C$ 
are evaluated exactly (cf. Assumption \ref{ass_IRGNM_GQnormevaluatedexactly}).

With $\varrho$ chosen sufficiently large such that \eqref{eq_IRGNM_rhobar} holds,
we define the quantities of interest as follows
\be{eq_IRGNM_ggn_I1234}
\begin{aligned}
I_1\colon&\  V\times Q\times V\times \R \to \R\,, &&
	(\uold,q,u,\beta) &&\mapsto \norm{C'(\uold)(u-\uold)+C(\uold)-g^\delta}_G^2 \\
	& && &&  \qquad + \frac 1 \beta \left(\norm{q-q_0}_Q^2+\norm{u-u_0}_V^2 \right)\\
I_2\colon&\ V\times V\to \R\,, && 
	(\uold,u) && \mapsto \norm{C'(\uold)(u-\uold)+C(\uold)-g^\delta}_G^2\\
I_3\colon&\ Q\times V\to \R\,,&&
	(\qold,\uold) &&\mapsto \norm{C(\uold) - g^\delta}_G^2 + \varrho\norm{A(\qold,\uold)-f}_{W^*}\\
I_4\colon&\ Q\times V \to \R\,, &&
	(q,u) && \mapsto \norm{C(u) - g^\delta}_G^2 + \varrho\norm{A(q,u)-f}_{W^*}
\end{aligned}
\ee
(cf. \eqref{eq_IRGNM_aao_I1234reduced}) and
\be{eq_IRGNM_ggn_I1234k}
\begin{aligned}
I_1^k &= I_1(\uoldk,\qk,\uk,\beta_k)\\
I_2^k &= I_2(\uoldk,\uk)\\
I_3^k &= I_3(\qoldk,\uoldk)\\
I_4^k &= I_4(\qk,\uk)\,,
\end{aligned}
\ee
(cf. \eqref{eq_IRGNM_aao_I1234k}),
where $\qoldk$, $\uoldk$ are fixed from the previous step and $\qk$, $\uk$ are coupled by the linearized state equation \eqref{eq_IRGNM_ggn_state} 
(or the third line of \eqref{eq_IRGNM_KKT_sqp} respectively) 
for $\qkminuseins = \qoldk$ and $\ukminuseins = \uoldk$.

Consistently, the discrete counterparts to \eqref{eq_IRGNM_ggn_I1234} and \eqref{eq_IRGNM_ggn_I1234k} 
are
\be{eq_IRGNM_ggn_I1234discrete}
\begin{aligned}
\Ieinsh\colon&\  V\times Q\times V\times \R \to \R\,, &&
	(\uold,q,u,\beta) &&\mapsto \norm{C'(\uold)(u-\uold)+C(\uold)-g^\delta}_G^2 \\
	& && &&  \qquad + \frac 1 \beta \left(\norm{q-q_0}_Q^2+\norm{u-u_0}_{V_{h_k}}^2 \right)\\
\Izweih\colon&\ V\times V\to \R\,, && 
	(\uold,u) && \mapsto I_2(\uold,u)\\
\Idreih\colon&\ Q\times V\to \R\,,&&
	(\qold,\uold) &&\mapsto \norm{C(\uold) - g^\delta}_G^2 + \varrho\norm{A(\qold,\uold)-f}_{W_{h_k}^*}\\
\Ivierh\colon&\ Q\times V \to \R\,, &&
	(q,u) && \mapsto \norm{C(u) - g^\delta}_G^2 + \varrho\norm{A(q,u)-f}_{W_{h_k}^*}
\end{aligned}
\ee
and
\be{eq_IRGNM_ggn_I1234kdiscrete}
\begin{aligned}
\Ieinshk &= \Ieinsh(\uoldk,\qhkk,\uhkk,\beta_k)\\
\Izweihk &= \Izweih(\uoldk,\uhkk)\\
\Idreihk &= \Idreih(\qoldk,\uoldk)\\
\Ivierhk &= \Ivierh(\qhkk,\uhkk)
\end{aligned}
\ee
(cf. \eqref{eq_IRGNM_aao_I1234kdiscrete}), where $\qoldk,\  \uoldk$ are fixed from the previous step, 
since (like in \eqref{eq_IRGNM_ggn_qoldkplus}) we set $\qoldkplus = \qhkk$ and $\uoldkplus = \uhkk$ 
at the end of each iteration step.

\begin{rem}\label{rem_IRGNM_Wnormestimate}
Here, as compared to \eqref{eq_IRGNM_aao_I1234k}, we have removed the $W^*$-norms in the 
definition of $I_1^{k}$ and $I_2^{k}$. 

The $W^*$-norm still appears in $I_3^k$, but only in connection
with the old iterative $(\qoldk, \uoldk)$, such that the only source of error in $I_3^k$ is the evaluation
of the $W^*$ norm. That means that with respect to Section \ref{subsec_IRGNM_ls}, we have replaced
the problematic expression 
\[
  \norm{E(q,u)}_{W^*} - \norm{E(\qh,\uh)}_{W_h^*}
\]
(cf. Section \ref{subsec_estimators_ls}) by an expression of the form
\be{errorE}
  \norm{E(\qold,\uold)}_{W^*} - \norm{E(\qold,\uold)}_{W_h^*}\,,
\ee
For the very typical case $W=V=H_0^1(\Omega)$ (see Section \ref{sec_IRGNM_numericalresults}), 
we can indeed estimate such an error using goal oriented error estimators:

Let $v\in V$, $v_h\in V_h$ solve the equations
\begin{align*}
  (\nabla v,\nabla \varphi)_{L^2(\Omega)} &= \dualV{E(\qold,\uold)}{\varphi} \quad \forall \varphi \in V\,,\\
  (\nabla v_h,\nabla \varphi)_{L^2(\Omega)} &= \dualV{E(\qold,\uold)}{\varphi} \quad \forall \varphi \in V_h\,,
\end{align*}
where $(.,.)_{L^2(\Omega)}$ denotes the scalar product in $L^2(\Omega)$ and $\dualV{.}{.}$ denotes the duality pairing between $V^*$ and $V$.
Then there holds 
\[
  \norm{E(\qold,\uold)}_{V^*} = \norm{\nabla v}_{L^2(\Omega)}\quad\text{and}\quad \norm{E(\qold,\uold)}_{V_h^*} = \norm{\nabla v_h}_{L^2(\Omega)}\,.
\]
We define the functional 
\[
  \Psi(v)\coloneqq \norm{\nabla v}_{L^2(\Omega)}
\]
and the Lagrangian 
\[
  L(v,w) \coloneqq \Psi(v) + \dualV{E(\qold,\uold)}{w} - (\nabla v,\nabla w)_{L^2(\Omega)}\,.
\]
Let $(v,w)$ and $(v_h,w_h)$ be continuous and discrete stationary points of $L$, i.e. 
 \begin{align}
 L_v'(v,w)(\varphi) 
 & = ({\normklein{\nabla v}_{L^2(\Omega)}})^{-1}  (\nabla v,\nabla\varphi)_{L^2(\Omega)}- (\nabla\varphi,\nabla w)_{L^2(\Omega)} =0 &&
 \forall \varphi\in V\,,\label{wequalsvcontinuous}\\
 L_w'(v,w)(\varphi) &= \dualV{E(\qold,\uold)}{\varphi} - (\nabla v,\nabla\varphi)_{L^2(\Omega)}=0 &&
 \forall \varphi\in V\,,\nonumber\\
 L_v'(v_h,w_h)(\varphi) 
 &= ({\normklein{\nabla v_h}_{L^2(\Omega)}})^{-1} (\nabla v_h,\nabla\varphi)_{L^2(\Omega)}
 - (\nabla\varphi,\nabla w_h)_{L^2(\Omega)}=0 &&
  \forall \varphi\in V_h\,,\label{wequalsvdiscrete}\\
 L_w'(v_h,w_h)(\varphi) &= \dualV{E(\qold,\uold)}{\varphi} - (\nabla v_h,\nabla\varphi)_{L^2(\Omega)} =0 &&
  \forall \varphi\in V_h\,. \nonumber
\end{align}
Then (by \eqref{wequalsvcontinuous} and \eqref{wequalsvdiscrete}) we have 
\be{wequalsvcontinuousAnddiscrete}
  w = ({\normklein{\nabla v}_{L^2(\Omega)}})^{-1} v \qquad \text{and} \qquad  w_h = ({\normklein{\nabla v_h}_{L^2(\Omega)}})^{-1} v_h\,.
\ee
For the error \eqref{errorE} then holds
\begin{align*}
  & \norm{E(\qold,\uold)}_{V^*} - \norm{E(\qold,\uold)}_{V_h^*}\\
  &\qquad = \Psi(v)-\Psi(v_h)\\ 
  &\qquad = \tfrac 12 L'(v_h,w_h)(v-\tilde {v_h},w-\tilde {w_h}) + R\\
  &\qquad = \tfrac 12 ({\normklein{\nabla v_h}_{L^2(\Omega)}})^{-1} (\nabla v_h,\nabla(v-\tilde{v_h}))_{L^2(\Omega)}
  -\tfrac 12  (\nabla(v-\tilde{v_h}),\nabla w_h)_{L^2(\Omega)}\\
  &\qquad \qquad + \tfrac 12 \dualV{E(\qold,\uold)}{w-\tilde{w_h}} - \tfrac 12 (\nabla v_h,\nabla(w-\tilde{w_h}))_{L^2(\Omega)}\\
  &\qquad = \tfrac 12 ({\normklein{\nabla v_h}_{L^2(\Omega)}})^{-1} (\nabla v_h,\nabla(v-\tilde{v_h}))_{L^2(\Omega)}
  -\tfrac 12 ({\normklein{\nabla v_h}_{L^2(\Omega)}})^{-1} (\nabla(v-\tilde{v_h}),\nabla v_h)_{L^2(\Omega)}\\
  &\qquad \qquad + \tfrac 12 \dualV{E(\qold,\uold)}{w-\tilde{w_h}} - \tfrac 12 (\nabla v_h,\nabla(w-\tilde{w_h}))_{L^2(\Omega)}\\
  &\qquad  =\tfrac 12 \dualV{E(\qold,\uold)}{w-\tilde{w_h}} - \tfrac 12 (\nabla v_h,\nabla(w-\tilde{w_h}))_{L^2(\Omega)}
\end{align*}
for arbitrary $\tilde{v_h},\tilde {w_h} \in V_h$, where $R$ is a third order remainder term
  (see e.g. \cite{BeckerRannacher,BeckerVexler}, Section \ref{subsec_estimatorsSQP}). 
  Please note that due to the relation \eqref{wequalsvcontinuousAnddiscrete} no additional system of equations has to be solved
  in order to obtain the additional variable $w_h$.  
  
Another way to deal with the discretization error in $I_3^k$ is the following: Tracking the upcoming 
convergence proof (cf. Theorem \ref{thm_IRGNM_conv_sqp}) the reader should realize that the discretization
for $I_{3,h}^k$ does not have to be the same as for $I_{1,h}^k$, $I_{2,h}^k$, such that $I_{3,h}^k$
could be evaluated on a very fine separate mesh, such that $\eta_3^k$ could be neglected. This alternative is
of course, more costly, but since everything else is still done on the adaptively refined (coarser) mesh, 
the proposed method could still lead to an efficient algorithm. 

%
%

The $W^*$-norm also appears in $I_4^{k}$, and unfortunately, in combination with the current $q$ and $u$,
which are subject to discretization, such that in principle we face the same situation as in the least squares formulation
from Section \ref{subsec_IRGNM_ls} (cf. Subsection \ref{subsec_estimators_ls}). Since, however, $\eta_4^k$ only appears in connection with the very
weak assumption $\eta_4^{k}\to0$ as $k\to\infty$ (cf. \eqref{eq_IRGNM_etacond2}), as in \cite{KKVV13}, we save ourselves the computational 
effort of computing an error estimator for $I_4^k$.
\end{rem}

Like in Section \ref{subsec_IRGNM_ls} we need the weak sequential closedness of $\bF$, i.e. Assumption \ref{ass_IRGNM_aao_Fweakseqclosed} 
and the following tangential cone condition (cf. Assumption \ref{ass_IRGNM_aao_tangcone}).

\begin{ass}\label{ass_IRGNM_cond2_sqp_AC}
There exist $0<c_{tc}<1$ and $\rho>0$ such that 
\begin{eqnarray}
\|C(u) - C(\bar{u})  - C'(u)(u - \bar{u})\|_{G} 
&\leq& c_{tc} \| C(u) - C(\bar{u})\|_{G}  
\nonumber\\
\|A(q,u) - A(\bar{q},\bar{u}) 
- A'_q(q,u)(q - \bar{q}) - A'_u(q,u)(u - \bar{u})\|_{W^*}
&\leq& 4c_{tc}^2 \| A(q,u) - A(\bar{q},\bar{u})\|_{W^*}  \,,
\nonumber
\end{eqnarray}
holds for all $(q,u),(\ol q,\ol u)\in \mathcal{B}_{\rho}(q_0,u_0) \subset Q\times V$ 
(cf. Assumption \ref{ass_IRGNM_aao_qdagger}).
\end{ass}

By means of Lemma \ref{lem_IRGNM_adjointbounded} and the Assumptions \ref{ass_IRGNM_aao_Fweakseqclosed},
\ref{ass_IRGNM_cond2_sqp_AC} 
we can now formulate a convergence result like in Theorems 1 and  3 
in \cite{KKVV13} and Theorem \ref{thm_IRGNM_conv_aao} here
for \eqref{eq_IRGNM_IRGNMsqp}. This can be done similarly to the proof of Theorem 3
in \cite{KKVV13}, replacing $F$ there by $\mathbf{F}$ according to \eqref{eq_IRGNM_bfF} and setting
\be{eq_IRGNM_SRchoice}
\begin{aligned}
\cS\left(\vektor{y_C}{y_A},\vektor{\tilde{y}_C}{\tilde{y}_A}\right) & = \|y_C-\tilde{y}_C\|_{G}^2+\varrho \|y_A-\tilde{y}_A\|_{W^*}\,,\\ 
\cR\left(\vektor{q}{u}\right)&=\|q-q_0\|_Q^2+\|u-u_0\|_V^2\,,\\
c_{\cS} &= 2\,,
\end{aligned}
\ee
there. For clarity of exposition we provide the full convergence proof (Theorem \ref{thm_IRGNM_conv_sqp}) without making use of the equivalence to \eqref{eq_IRGNM_aao_var_r} here. 
Only for the convergence rates result Theorem \ref{thm_IRGNM_rates_sqp} we refer to Theorem 4 in \cite{KKVV13} with \eqref{eq_IRGNM_SRchoice} and the equivalence to \eqref{eq_IRGNM_aao_var_r}. 
So in the proof of Theorem \ref{thm_IRGNM_conv_sqp} we will not use minimality wrt \eqref{eq_IRGNM_aao_var_r} but only wrt the original formulation \eqref{eq_IRGNM_IRGNMsqp}, \eqref{eq_IRGNM_ggn_state} (actually we are using KKT points instead of minimizers, but this make no real difference due to convexity of the problem).

%
\begin{thm} \label{thm_IRGNM_conv_sqp}
Let the Assumptions \ref{ass_IRGNM_aao_qdagger}, \ref{ass_auinv}, \ref{ass_IRGNM_GQnormevaluatedexactly}, \ref{ass_IRGNM_aao_Fweakseqclosed} and \ref{ass_IRGNM_cond2_sqp_AC}
with $c_{tc}$ sufficiently small be satisfied and let Assumption \ref{ass_IRGNM_tauthetacond} hold.
For the quantities of interest \eqref{eq_IRGNM_ggn_I1234k} and \eqref{eq_IRGNM_ggn_I1234kdiscrete},
let, further, the estimate \eqref{eq_IRGNM_intcond} hold with $\eta_i$ satisfying Assumption \ref{ass_IRGNM_etacond}. 

Then with $\beta_k$, $h=h_k$ fulfilling \eqref{eq_IRGNM_inexNewton}, $k_*$ selected according to \eqref{eq_IRGNM_stop}, and 
$(\qhkk,\uhkk)$ defined as the primal part of a KKT point of 
\eqref{eq_IRGNM_IRGNMsqpdiscrete}, \eqref{eq_IRGNM_ggn_statediscrete} there holds
 \begin{itemize}
 \item[(i)]
 \beq\label{eq_IRGNM_ineq0_sqp}
  \|\qhkk -q_0\|^2+\|\uhkk -u_0\|^2 \leq \|q^\dag -q_0\|^2+\|u^\dag -u_0\|^2 \qquad \forall 0\le k\leq k_*\,;
 \eeq
 \item[(ii)] $k_*$ is finite\,;
 \item[(iii)] $(\qhkstern,\uhkstern)=(q_{h_{k_*(\delta)}}^{k_*(\delta),\delta},u_{h_{k_*(\delta)}}^{k_*(\delta),\delta})$ converges (weakly) subsequentially to a solution of \eqref{eq_IRGNM_Fqug} as 
 $\delta \to 0$
in the sense that it has a weakly convergent subsequence and each
weakly convergent subsequence converges strongly to a solution of \eqref{eq_IRGNM_Fqug}.
If the solution $(q^\dagger,u^\dagger)$ to \eqref{eq_IRGNM_Fqug} is unique, then $(\qhkstern,\uhkstern)$ 
converges strongly to $(q^\dagger,u^\dagger)$  as $\delta\to0$.
\end{itemize}
\end{thm}
We mention in passing that this is a new result also in the continuous case $\eta_i^k=0$.
\begin{proof}
\begin{itemize}
\item[(i):] 
We will prove \eqref{eq_IRGNM_ineq0_sqp} by induction. The base case $k=0$ is trivial.
To carry out the induction step, we assume that 
\be{eq_IRGNM_induktionsvoraussetzung}
  \|\qhkminuseinsk -q_0\|_Q^2+\|\uhkminuseinsk-u_0\|_V^2 \leq \|q^\dag -q_0\|_Q^2+\|u^\dag -u_0\|_V^2 
  \qquad \forall 1\le k\le k_*
\ee
holds. We consider a continuous step emerging from discrete $\qoldk = \qhkminuseinsk$, $\uoldk =\uhkminuseinsk$ 
(cf. Figure \ref{fig:iteration}),
i.e. let $(\qk,\uk)$ be a solution to \eqref{eq_IRGNM_IRGNMsqp} for 
$\qkminuseins = \qoldk = \qhkminuseinsk$. 
Then the KKT conditions $\L'(\qk,\uk,\zk)=0$ 
(cf. \eqref{eq_ggn_Lz}-\eqref{eq_ggn_Lq}) imply
\begin{align*}
0&= \scalarG{\Chk(\uk-\uhkminuseinsk) +\rg}{\Chk \delta u}
+ {\frac 1 {\beta_k}} \left[\scalarQ{\qk -q_0}{\delta q} +\scalarV{\uk -u_0}{\delta u}\right]\\
& \quad +\frac12 \scalarWstern{\Lhk \delta q+\Khk \delta u}{\pk}
\end{align*}	
for all $\delta q \in Q$ and $\delta u \in V$, where we have used the same abbreviations as in \eqref{eq_IRGNM_ls_abbrevKL} and
\eqref{eq_IRGNM_abbrevCr}, as well as $\pk$ defined by \eqref{eq_IRGNM_pdef}.

Setting $\delta q=\qk-q^\dag$, $\delta u=\uk-u^\dag$, this yields
\begin{align*}
0&= \| \Chk(\uk-\uhkminuseinsk) + \rg\|_G^2\\
&\quad -\scalarG{\Chk(\uk-\uhkminuseinsk) +\rg}{\Chk(u^\dag-\uhkminuseinsk) +\rg}\\
&\quad  + {\frac 1 {\beta_k}} \|\qk -q_0\|_Q^2
-{\frac 1 {\beta_k}} \scalarQ{ \qk -q_0}{q^\dag-q_0}
  + {\frac 1 {\beta_k}} \|\uk -u_0\|^2
-{\frac 1 {\beta_k}} \scalarV{\uk -u_0}{u^\dag-u_0}\\
&\quad-\frac12 \scalarWstern{\Lhk(q^\dag-\qhkminuseinsk) + \Khk(u^\dag-\uhkminuseinsk)+\rf}{\pk}\,,
\end{align*}
where we have used the fact that $(\qk,\uk)$ satisfies the linearized state equation \eqref{eq_IRGNM_ggn_state}, 
i.e. 
\[
\Lhk(\qk-\qhkminuseinsk)\Khk(\uk-\uhkminuseinsk)+\rf=0\,.
\]
Hence by Cauchy-Schwarz and the fact that $ab\le \frac12 a^2+\frac12 b^2$ for all $a,b\in \R$ 
\begin{align*}
I_1^k 
&\le \normGklein{\Chk(\uk-\uhkminuseinsk) +\rg}\normGklein{\Chk(u^\dag-\uhkminuseinsk) +\rg}\\
&\quad + {\frac 1 {\beta_k}} \normQklein{ \qk -q_0}\normQklein{q^\dag-q_0} +{\frac 1 {\beta_k}} \normVklein{\uk -u_0}\normVklein{u^\dag-u_0}\\
&\quad +  \frac12 \normWstern{\Lhk(q^\dag-\qhkminuseinsk) + \Khk(u^\dag-\uhkminuseinsk)+\rf}\normWsternklein{\pk}\\
&\le \frac 12 \normGklein{\Chk(\uk-\uhkminuseinsk) +\rg}^2 + \frac 12 \normGklein{\Chk(u^\dag-\uhkminuseinsk) +\rg}^2\\
&\quad + {\frac 1 {2\beta_k}} \normQklein{ \qk -q_0}^2 + {\frac 1 {2\beta_k}} \normQklein{q^\dag -q_0}^2 
 + {\frac 1 {2\beta_k}} \normVklein{\uk -u_0}^2 + {\frac 1 {2\beta_k}} \normVklein{u^\dag -u_0}^2\\
&\quad +  \frac12 \normWstern{\Lhk(q^\dag-\qhkminuseinsk) + \Khk(u^\dag-\uhkminuseinsk)+\rf}\normWsternklein{\pk}\\
&= \frac 12 I_1^k + \frac 12 \normGklein{\Chk(u^\dag-\uhkminuseinsk) +\rg}^2\\
&\quad + {\frac 1 {2\beta_k}} \normQklein{q^\dag -q_0}^2 
 + {\frac 1 {2\beta_k}} \normVklein{u^\dag -u_0}^2\\
&\quad +  \frac12 \normWstern{\Lhk(q^\dag-\qhkminuseinsk) + \Khk(u^\dag-\uhkminuseinsk)+\rf}\normWsternklein{\pk}\,,
\end{align*}
which dividing by 2 and applying Lemma \ref{lem_IRGNM_adjointbounded} with \eqref{eq_IRGNM_rhobar},
and \eqref{eq_IRGNM_induktionsvoraussetzung} leads to 
\begin{align}
I_1^k &\le \normGklein{\Chk(u^\dag-\uhkminuseinsk) +\rg}^2 + {\frac 1 {\beta_k}} \normQklein{q^\dag -q_0}^2 
 + {\frac 1 {\beta_k}} \normVklein{u^\dag -u_0}^2\nonumber\\
&\quad + \normWstern{\Lhk(q^\dag-\qhkminuseinsk) + \Khk(u^\dag-\uhkminuseinsk)+\rf}\normWsternklein{\pk}\nonumber\\
&\le \normGklein{C'(\uhkminuseinsk)(u^\dag-\uhkminuseinsk) + C(\uhkminuseinsk)-g^\delta}^2 + {\frac 1 {\beta_k}} \normQklein{q^\dag -q_0}^2 
 + {\frac 1 {\beta_k}} \normVklein{u^\dag -u_0}^2\nonumber\\
&\quad + \varrho \normWstern{A_q'(\qhkminuseinsk)(q^\dag-\qhkminuseinsk) + A_u'(\uhkminuseinsk)(u^\dag-\uhkminuseinsk)+A(\qhkminuseinsk,\uhkminuseinsk)-A(q^\dag,u^\dag)}
\label{eq_IRGNM_helpestimate1}
\end{align}
for all $k<k_*$.
%

The rest of the proof basically follows the lines of the proof of Theorem 3
in \cite{KKVV13}
with the choice \eqref{eq_IRGNM_SRchoice}, but for convenience of the reader we will follow through the proof
anyway. 

Using the fact that $(a+b)^2\le 2a^2+2b^2$ for all $a,b\in\R$ and Assumption \ref{ass_IRGNM_cond2_sqp_AC} from 
\eqref{eq_IRGNM_helpestimate1} we get 
\begin{align*}
I_1^k &\le 2c_{tc}^2\normGklein{C(\uhkminuseinsk)-C(u^\dag)}^2 + 2\delta^2+ {\frac 1 {\beta_k}} \normQklein{q^\dag -q_0}^2 
 + {\frac 1 {\beta_k}} \normVklein{u^\dag -u_0}^2\\
&\quad + 4c_{tc}^2 \varrho \normWstern{A(\qhkminuseinsk,\uhkminuseinsk)-A(q^\dag,u^\dag)} \\
&\le 4c_{tc}^2\normGklein{C(\uhkminuseinsk)-g^\delta}^2  +  2(1+2c_{tc}^2) \delta ^2+ {\frac 1 {\beta_k}} \normQklein{q^\dag -q_0}^2 
 + {\frac 1 {\beta_k}} \normVklein{u^\dag -u_0}^2\\
&\quad + 4c_{tc}^2 \varrho \normWstern{A(\qhkminuseinsk,\uhkminuseinsk)-f} \\
&\le 4c_{tc}^2 I_3^k + \frac {2(1+2c_{tc}^2)}{\tau^2} I_{3,h}^k + {\frac 1 {\beta_k}} \normQklein{q^\dag -q_0}^2 
 + {\frac 1 {\beta_k}} \normVklein{u^\dag -u_0}^2\\
&\le 2\left(2c_{tc}^2 + \frac {1+2c_{tc}^2}{\tau^2}\right) I_{3,h}^k +4c_{tc}^2\eta_3^k + {\frac 1 {\beta_k}} \left( \normQklein{q^\dag -q_0}^2 
 + \normVklein{u^\dag -u_0}^2 \right)
\end{align*}
for all $k<k_*$. 
This together with \eqref{eq_IRGNM_inexNewton} and 
the fact that $I_{1,h}^k = I_{2,h}^k + \frac 1 {\beta_k} \left(\normQklein{\qhk-q_0}^2 + \normVklein{\uhk-u_0}^2\right)$
yields
\begin{align*}
&\tilde{\underline\theta} I_{3,h}^{k} 
+ {\frac 1 {\beta_k}} (\|\qhkk -q_0\|^2+\|\uhkk -u_0\|^2)\\
& \le I_{2,h}^k + {\frac 1 {\beta_k}} (\|\qhkk -q_0\|^2+\|\uhkk -u_0\|^2)\\
& \le I_1^k + \eta_1^k\\
& \le 2\left(2c_{tc}^2+\frac{(1+2c_{tc})^2}{\tau^2}\right) I_{3,h}^{k}
+ {\frac 1 {\beta_k}} (\|q^\dag -q_0\|^2+\|u^\dag -u_0\|^2) 
+ \eta_1^{k} + 4c_{tc}^2 \eta_3^{k}\,.
\end{align*}
hence by \eqref{eq_IRGNM_etacond1} we get \eqref{eq_IRGNM_ineq0_sqp}.
\item[(ii):]
By the triangle inequality as well as \eqref{eq_IRGNM_inexNewton},
Assumption \ref{ass_IRGNM_cond2_sqp_AC} and the fact that $(\qk,\uk)$ satisfies the linearized state equation \eqref{eq_IRGNM_ggn_state}, 
we have
\begin{align*}
I_4^k &= \normGklein{C(\uk)-g^\delta}^2 + \varrho \normWsternklein{A(\qk,\uk)-f}\\
&\le 2\normGklein{C'(\uoldk)(\uk-\uoldk) + C(\uoldk) -g^\delta}^2
    + 2\normGklein{C'(\uoldk)(\uk-\uoldk)+C(\uoldk)-C(\uk)}^2\\
&\quad  + \varrho \normWsternklein{A'_q(\qoldk,\uoldk)(\qk-\qoldk) 
+A_u'(\qoldk,\uoldk)(\uk-\uoldk) + A(\qoldk,\uoldk)-f}\\
&\quad + \varrho \normWsternklein{A'_q(\qoldk,\uoldk)(\qk-\qoldk)
+A_u'(\qoldk,\uoldk)(\uk-\uoldk) + A(\qoldk,\uoldk)-A(\qk,\uk)}\\
&\le 2 I_2^k + 2 c_{tc}^2 \normGklein{C(\uk)-C(\uoldk)}^2 + 4 c_{tc}^2\varrho \normWsternklein{A(\qoldk,\uoldk)-A(\qk,\uk)}\\
&\le 2 I_2^k + 4 c_{tc}^2 \left( \normGklein{C(\uk)-g^\delta}^2 + \normGklein{C(\uoldk)-g^\delta}^2 \right)\\
&\quad +  4 c_{tc}^2\varrho \left(\normWsternklein{A(\qoldk,\uoldk)-f} + \normWsternklein{A(\qk,\uk)-f}\right)\\
&\le 2\left(\tilde{\overline\theta} I_{3,h}^{k} +\eta_2^{k}\right) 
+ 4c_{tc}^2 (I_4^k+I_3^k)\,,
\end{align*}
which implies
\[
  I_4^k \le \frac 1 {1-4c_{tc}^2}\left(2 {\tilde{\overline\theta}} I_{3,h}^{k} + 2 \eta_2^k + 4c_{tc}^2 I_3^k \right)\,.
\]
From this, using \eqref{eq_IRGNM_intcond} and \eqref{eq_IRGNM_etacond3} we can deduce exactly as in the proof of Theorem 3 (ii) in \cite{KKVV13} that 
\be{estI4_2}
  I_{4,h}^k \le c_2^{k}I_{4,h}^0+\sum_{j=0}^{k-1}c_2^j a^{k-j}\,.
\ee
with 
\be{ai}
  a^i\coloneqq \frac 1 {1-4c_{tc}^2} \left((2 {\tilde{\overline\theta}} + 4c_{tc}^2)r^i + 2 \eta_2^i + 4c_{tc}^2 \eta_3^i \right) + \eta_4^i \quad\forall i\in \{1,2,\dots,k\}\,.
\ee
So since the right hand side of \eqref{estI4_2} tends to zero as $k\to\infty$, $I_{4,h}^k$ and therewith $I_{3,h}^k$ (cf. \eqref{eq_IRGNM_etacond3}) eventually has to fall below $\tau^2\delta^2$ for some finite index $k$. 

\item[(iii):]
With \eqref{eq_delta}, \eqref{eq_IRGNM_intcond}, \eqref{eq_IRGNM_etacond2} and the definition of $k_*$, we have 
\be{eq_IRGNM_aao_Fconv}
\begin{aligned}
  \normGklein{C(\uoldkstern)-g}^2 + \varrho \normWsternklein{A(\qoldkstern,\uoldkstern)-f}
  &\le 2 I_3^{k_*} + 2\delta^2\\ 
  &\le 2\left( I_{3,h}^{k_*} + \eta_3^k + \delta^2 \right)\\
  &\le 2\left( (1+c_1)I_{3,h}^{k_*} + \delta^2 \right) \\ 
  &\le 2\delta^2\left( (1+c_1)\tau^2 + 1\right) \to 0
\end{aligned}
\ee
as $\delta \to 0$. Thus, due to (ii) \eqref{eq_IRGNM_ineq0_sqp}  $(\qoldkstern,\uoldkstern) = (q^{\delta,k_*-1}_{h_{k_*-1}},u^{\delta,k_*-1}_{h_{k_*-1}})$ 
has a weakly convergent subsequence $\left((\qoldksterndeltal,\uoldksterndeltal)\right)_{l\in \N}$ 
 and with Assumption \ref{ass_IRGNM_aao_Fweakseqclosed} and \eqref{eq_IRGNM_aao_Fconv} the limit of every weakly 
convergent subsequence is a solution to \eqref{eq_IRGNM_Fqug}.
Strong convergence of any weakly convergent subsequence again follows by a
standard argument like in \cite{KKVV13} using \eqref{eq_IRGNM_ineq0_sqp}.
\end{itemize}
\end{proof}
\begin{cor}\label{cor_IRGNM_adjointbounded}
The sequence $(\zk)_{k\in\N,k\le k_*}$ is bounded, i.e.
\[
 \bar{\varrho}=\sup_{k\leq k_*} \|\zk\|_W  \le \cadj (\norm{q^\dag-q_0}^2+\norm{u^\dag-u_0}^2+1)
\]
\end{cor}
\begin{proof}
The assertion follows directly from Theorem \ref{thm_IRGNM_conv_sqp} (i) and Lemma \ref{lem_IRGNM_adjointbounded}.
\end{proof}

The convergences rates from Theorem 4 in \cite{KKVV13} 
also hold for the all-at-once formulation \eqref{eq_IRGNM_Fqug}, due to equivalence with \eqref{eq_IRGNM_aao_var_r} which we 
formulate in the following theorem. 

Instead of source conditions we use variational source conditions (cf., e.g., \cite{Flemming2010,HohageWerner,KH10,PoeschlDiss}) due to the nonquadratic penalty term in \eqref{eq_IRGNM_aao_var_r}.

\begin{ass}\label{ass_IRGNM_vi_aao}
Let 
\begin{eqnarray}\label{eq_IRGNM_source2_sqp}
\lefteqn{
|\scalarQ{q^\dag-q_0}{q -q^\dag}+\scalarV{u^\dag-u_0}{u -u^\dag}|}
\nonumber\\
&\leq& c \, \sqrt{\|q-q^\dag\|_Q^2+\|u-u^\dag\|_V^2} 
\kappa\left( \frac{\|C(u)-C(u^\dag)\|_{G}^2+\varrho \|A(q,u)-A(q^\dag,u^\dag)\|_{W^*}}{\|q-q^\dag\|_Q^2
+\|u-u^\dag\|_V^2}\right) \,, \\ 
&&(q,u)\in \mathcal{D}(A)\,, \ u\in \mathcal{D}(C)\,,
\end{eqnarray}
with $\varrho$ sufficiently large (cf. \eqref{eq_IRGNM_rhobar}) {  and independent from $q,u$},
hold with some $\kappa\colon \R^+\to\R^+$ such that $\kappa^2$ is strictly monotonically increasing on
$(0, \|\bF(q^\dag,u^\dag)\|_{Q\times V}^2]$,
$\phi$ defined by $\phi^{-1}(\lambda)=\kappa^2(\lambda)$ is convex and $\psi$ defined by
  $\psi(\lambda)=\kappa(\lambda)\sqrt{\lambda}$ is  strictly monotonically increasing on
  $(0,\|\bF(q^\dag,u^\dag)\|_{Q\times V}^2]$. 
\end{ass}

\begin{thm} \label{thm_IRGNM_rates_sqp}
Let the conditions of Theorem \ref{thm_IRGNM_conv_sqp} and additionally
the variational inequality Assumption \ref{ass_IRGNM_vi_aao} be fulfilled.
 
Then there exists a $\bar{\delta}>0$ and a constant $\bar{C}>0$ independent of $\delta$ such that 
for all $\delta\in (0,\bar{\delta}]$ß
the convergence rates
\begin{equation}\label{eq_IRGNM_rates_sqp}
   \normQklein{\qold^{k_*}-\qdag}^2+ \normVklein{\uold^{k_*}-u^\dag}^2
  = {\cal O}\left(\frac {\delta^2} {\psi^{-1}(\bar{C}\delta)}\right)\,,
   \end{equation}
with  $\qold^{k_*}= q_{\beta_{k_*-1},h_{k*-1}}^{\delta,k_*-1}$,
      $\uold^{k_*} = u_{\beta_{k_*-1},h_{k*-1}}^{\delta,k_*-1}$
are obtained.
\end{thm}
\begin{proof} With \eqref{eq_IRGNM_SRchoice} the rate follows directly from Theorem 4 in \cite{KKVV13} 
due to Theorem \ref{thm_IRGNM_conv_sqp} 
(especially \eqref{eq_IRGNM_ineq0_sqp})
and \eqref{eq_IRGNM_aao_Fconv}.
\end{proof}

\begin{rem}\label{rem_IRGNM_estu}
In fact, no regularization of the $u$ part would be needed for proving just stability of the single Gauss-Newton steps,
since by Assumption \ref{ass_auinv} the terms 
$\varrho \|A'_q(\qkminuseins ,\ukminuseins )(q-\qkminuseins )
+A'_u(\qkminuseins ,\ukminuseins )(u-\ukminuseins )
+A(\qkminuseins ,\ukminuseins )-f\|_{W^*}$ 
and ${\frac 1 {\beta_k}} \|q-q_0\|_Q^2$ in \eqref{eq_IRGNM_aao_var_r} as regularization term together ensure weak compactness of the level sets of the Tikhonov functional (cf. Item 6 in Assumption 2 in \cite{KKVV13}). 
However, we require even uniform boundedness of $\uhkk$  in order to uniformly bound the dual variable and come up with a penalty parameter $\varrho$ that is independent of $k$, cf. the discrete version of Lemma \ref{lem_IRGNM_adjointbounded}.
Using 
the equality constraint \eqref{eq_IRGNM_ggn_state} (for $\qkminuseins = \qhkminuseinsk$)
together with the tangential cone condition Assumption \ref{ass_IRGNM_cond2_sqp_AC} for $c_{tc}^2<\frac 1 8$
would only enable to bound $A(\qk,\uk)-f$:
\begin{align*}
&\normWsternklein{A(\qk, \uk)-f} \\
&\qquad = \normWsternklein{\Lhk(\qk-\qhkminuseinsk) +\Khk(\uk-\uhkminuseinsk)+ A(\qhkminuseinsk,\uhkminuseinsk) - A(\qk, \uk)}\\
&\qquad \le 4c_{tc}^2 \normWsternklein{A(\qhkminuseinsk,\uhkminuseinsk)-A(\qk,\uk)}\\
&\qquad  \le 4c_{tc}^2 \normWsternklein{A(\qhkminuseinsk,\uhkminuseinsk)-f} + 4c_{tc}^2 \normWsternklein{A(\qk,\uk)-f}\,,
\end{align*}
such that
\[
\normWsternklein{A(\qk, \uk)-f} \le \frac {4c_{tc}^2} {1-4c_{tc}^2} \normWsternklein{A(\qhkminuseinsk,\uhkminuseinsk)-f}\,.
\] 
However, without error estimators on the difference between 
$\|A(\qk,\uk)-f\|_{W^*}$ and its discretized version, this does not give a recursion 
\[
	\|A(\qhkk,\uhkk)-f\|_{W^*}\leq c\|A(\qhkminuseinsk,\uhkminuseinsk)-f\|_{W^*}
\]
(from which, by uniform boundedness of $\qhkk$ and Assumption \ref{ass_IRGNM_unifbyA} we could conclude uniform boundedness of $\uhkk$).
 
Thus, in order to obtain uniform boundedness of $\uhkk$ we introduce the term ${\frac 1 {\beta_k}} \|u-u_0\|_V^2$ here for theoretical purposes.
For our practical computations we will assume that the error by discretization between
$\|A(\qk,\uk)-f\|_{W^*}$ and $\|A(\qhkk,\uhkk)-f\|_{W^*}$ is small enough so that the mentioned gap in this argument for uniform boundedness of $\uhkk$ can be neglected and the part  ${\frac 1 {\beta_k}} \|u-u_0\|_V^2$ of the regularization term is omitted.
\end{rem}

\subsection{Computation of the error estimators}\label{subsec_estimatorsSQP}

Since -- different to \cite{KKVV13} -- $\uold$ ist not subject to new discretization in the $k$th step here, the computation of the error estimators is easier and can be done exactly as in
\cite{GKV} and \cite{KKV10}. Thus we omit the arguments $\qold$ and $\uold$ in the quantities of
interest in this subsection and 
we also omit the iteration
index $k$ and the explicit dependence on $\beta$.

\textbf{Error estimator for $I_1$:}
We consider
\[
   I_1(q,u)=\norm{C'(\uold)(u-\uold)+C(\uold)-g^\delta}_G^2
   + \alpha (\norm{q-q_0}_G^2+\norm{u-u_0}_V^2)
\]

and the Lagrange functional
\beq\label{eq_IRGNM_LagrangianSQP}
   \L(q,u,z)\coloneqq I_1(q,u)+h(z)-B(q,u)(z)\,,
\eeq
with $h \in W^*$ and $B(q,u)\in W^*$ defined as
\[
   h\coloneqq f-A(\qold,\uold)-A_q'(\qold,\uold)(\qold) - A_u'(\qold,\uold)(\uold)
\]
and
\[
   B(q,u)\coloneqq A_q'(\qold,\uold)(q)+A_u'(\qold,\uold)(u)\,.
\]
There holds a similar result to Proposition 1 in \cite{KKVV13}(see also \cite{GKV}), which
allows to estimate the difference $I_1(q,u)-I_1(\qh,\uh)$ by computing a discrete stationary point
$\xh=(\qh,\uh,\zh)\in X_h=Q_h\times V_h\times W_h$ of $\L$. This is done by solving the equations
\begin{align}
  \zh\in W_h:\qquad & A_u'(\qold,\uold)(du)(\zh) = I_{1,u}'(\qh,\uh)(du) &\forall du\in V_h \label{eq_statpoint1}\\
  \uh\in V_h: \qquad & A_q'(\qold,\uold)(\qh)(dz)+A_u'(\qold,\uold)(\uh)(dz) = h(dz)
  &\forall dz\in W_h\label{eq_statpoint2}\\
  \qh\in Q_h: \qquad & I_q'(\qh,\uh)(dq) = A_q'(\qold,\uold)(dq)(\zh)  &\forall dq\in Q_h\label{eq_statpoint3}
\end{align}
Then the error estimator $\eta_1$ for $I_1$ can be computed as
\begin{equation}\label{eq_IRGNM_eta1}
  I_1 - I_{1,h} =  I_1(q,u)-I_2(\qh,\uh) \approx \frac 12 \L'(\xh)(\pi_h\xh-\xh) = \eta_1
\end{equation}
(cf. \cite{KKVV13,KKV10,GKV}).

{ 
\begin{rem}
 Please note that the equations \eqref{eq_statpoint1}-\eqref{eq_statpoint3} are solved anyway in the process of solving the optimization problem \eqref{eq_IRGNM_IRGNMsqp}, \eqref{eq_IRGNM_ggn_state}.
\end{rem}
}

\textbf{Error estimator for $I_2$:}
The computation of the error estimator for $I_2$ can be done similarly to the computation of
$\eta_2$ in \cite{KKVV13} (or $\eta^I$ in \cite{GKV}) by means of
the Lagrange functional $\L$. We consider
\[
   I_2(u)\coloneqq\norm{C'(\uold)(u-\uold)+C(\uold)-g^\delta}_G^2
\]
and compute a discrete stationary point $\yh\coloneqq (\xh, \xeinsh)\in X_h\times X_h$ of the auxiliary
Lagrange functional
\[
   \M(y)\coloneqq I_2(u)+\L'(x)(x_1)
\]
by solving the equations
\begin{align}
   \xh\in X_h: \qquad & \L'(\xh)(dx_1) = 0 &\forall dx_1\in X_h\nonumber\\
   \xeinsh \in X_h:\qquad &\L''(\xh)(\xeinsh,dx) = -I_2'(\uh)(du) &\forall dx\in X_h\label{eq_statpointM}\\
\end{align}
(with $dx=(dq,du,dz)$). Then we compute the error estimator for $I_2$ by
\[
  \eta_2\coloneqq \frac 12 \M'(\yh)(\pi_h\yh-\yh)\approx I_2(u)-I_2(\uh) = I_2-I_{2,h}\,.
\]
{
 
\begin{rem}
 To avoid the computation of second order information in \eqref{eq_statpointM} we would like to refer to \cite{BeckerVexler}, where 
 \eqref{eq_statpointM} is replaced by an approximate equation of first order. 
\end{rem}
}

\textbf{Error estimator for $I_3$:}
In Remark \ref{rem_IRGNM_Wnormestimate}, we already mentioned that the $W^*$ norm in $I_3$ can be evaluated on a
separate very fine mesh, so that we will neglect the difference between $\norm{A(\qold,\uold)-f}_{W^*}$ and
$\norm{A(\qold,\uold)-f}_{W_h^*}$. This implies that we do not need to compute the error estimator
$\eta_3$, since $I_3 = I_{3,h}$, so that \eqref{eq_IRGNM_etacond3},  and the first part of \eqref{eq_IRGNM_etacond2}
is trivially fullfilled.

\textbf{Error estimator for $I_4$:}
We also mentioned in Remark \ref{rem_IRGNM_Wnormestimate} that we will not compute $\eta_4$, as the error
$|I_4-I_{4,h}|$ needs to be controlled only through the very weak assumption
$\eta_4^{k}\to0$ as $k\to\infty$ (cf. \eqref{eq_IRGNM_etacond2}), which in practice we will simply make sure by altogether decreasing the mesh size in the course of the iteration. 

\subsubsection{Algorithm}\label{sec_IRGNM_aao_alg}

Since we only know about the existence of an upper bound $\ol\varrho$ of $\normWklein{\zk}$ (cf.
Corollary \ref{cor_IRGNM_adjointbounded}), but not its value, we choose $\varrho$ (cf. \eqref{eq_IRGNM_rhobar}) heuristically, i.e. in each iteration step
we set $\varrho = \varrho_k = \max\{\varrho_{k-1}\,,\normklein{z_{h_k}^k}_{W_{h_k}}\}$ for the discrete counterpart $z_{h_k}^k$ of $\zk$. 

\begin{rem}
Theoretically one should use $\varrho = \normklein{z_{h_k}^k}_{W_{H}}$ on a very fine discretization $H$ in order to get a better approximation to $\normWklein{\zk}$. However, since we only need the correct order of magnitude and not the exact value, we just 
use the current mesh $h_k$. 
\end{rem}

In view of Remark \ref{rem_IRGNM_estu} we omit the  part  ${\frac 1 {\beta_k}} \|u-u_0\|_V^2$ of the regularization term.
Also, as motivated in Section \ref{subsec_estimatorsSQP}, we assume $\eta_3^k=0$ for all $k$, such that
we neither compute $\eta_3$ nor $\eta_4$.

Thus we only check for the condition 
\beq\label{eq_IRGNM_etacond1_algo}
\eta_1^{k} \leq
\left({\tilde{\underline\theta}} - 2\left(2c_{tc}^2 + \frac{(1+2c_{tc})^2}{\tau^2}\right)\right)
I_{3,h}^{k}
\eeq
on $\eta_1^k$ in Assumption \ref{ass_IRGNM_etacond}. 
{
 
For simplicity, we evaluate $I_{3,h}^k$ on the current mesh instead of a very fine mesh as explained in Remark \ref{rem_IRGNM_Wnormestimate}. 
}

For computing $\beta_k$, $h=h_k$ fulfilling \eqref{eq_IRGNM_inexNewton}, we can resort to the Algorithm from \cite{GKV}, which also contains refinement with respect to the quantity of interest $I_{2,h}^k$ and repeated solution of 
                   \be{eq_IRGNM_optprobinalg}
			\min_{(q,v)\in Q_{h} \times V_{h}}
			\|C'(\uoldhk)(v)+C(\uoldhk)-g^\delta\|_G^2+{\frac 1 {\beta_k}}\|q-q_0\|_Q^2
		    \ee
		    \[
		      \text{s.t.}\qquad A'_u(\qoldk,\uoldhk)(v)(\varphi)+A'_q(\qoldk,\uoldhk)(q-\qoldk)(\varphi)
					+ A(\qoldk,\uoldhk)(\varphi) - f(\varphi)=0
		      \qquad \forall \varphi\in W_{h}\,
		    \]
for $(q,v)\in Q_{h} \times V_{h}$.

The presented Generalized Gauss-Newton formulation can be implemented according to the following 
Algorithm \ref{alg_IRGNM_GGN}.
  \begin{algorithm}{Generalized Gauss-Newton Method}
    \begin{algorithmic}[1]\label{alg_IRGNM_GGN}
      \STATE Choose $\tau$, $\tau_\beta$, ${\tilde \tau}_\beta$, $\uttheta$, $\ottheta$ such that $ 0 < \uttheta \le \ottheta <1$ and
  	  Assumption \ref{ass_IRGNM_tauthetacond} holds.
  	  $\tilde\theta = (\uttheta + \ottheta) / 2$ 
           and $\max\{1\,,\tilde{\tau}_\beta\} <\tau_\beta\le \tau$,
  	  and choose $c_1$, $c_2$ and $c_3$, such that the second part of
  	  \eqref{eq_IRGNM_etacond3} is fulfilled.
      \STATE Choose a discretization $h=h_0$ and starting value  $\qhnull (= \qhnullnull)$
        (not necessarily coinciding with $q_0$ in the regularization term)
        and set $\qoldnull=\qhnull$.  
      \STATE  Choose starting value  $\uhnull(=\uhnullnull)$ (e.g. by solving the PDE 
 	$A(\qoldnull,\uoldnull) = f$)
        and set $\uoldnull=\uhnull$.    
      \STATE Compute the adjoint state $\zhnull(=\zhnullnull)$ 
  	  (see \eqref{eq_ggn_Lu}),
  	  evaluate $\norm{\zhnull}_{W_h}$, set $\varrho_0=\normklein{\zhnull}_{W_h}$.
  	  and evaluate $I_{3,h}^0$ (cf. \eqref{eq_IRGNM_aao_I1234kdiscrete}).
      \STATE Set $k=0$ and $h=h^1_0 =h_0$.
      \WHILE {$I_{3,h}^k>\tau^2\delta^2$} \label{refinewhile}
 	  \STATE Set $h = h^1_k$. 
  	  \STATE Solve the optimization problem \eqref{eq_IRGNM_optprobinalg}
  	  \STATE Set $h^2_k = h^1_k$ and $\delta_\beta^2=\tilde \theta I_{3,h}^k$.
   	   \IF{$I_{2,h}^k > \left(\tau_\beta^2 + \frac {{\tilde\tau}_\beta^2} 2\right)\delta_\beta^2$}
 		    \STATE With $\qoldk$, $\uoldk$ fixed, 
 		    apply the Algorithm from \cite{GKV} 
   		    (with quantity of interest $I_2^k$ and noise level $\delta_\beta^2=\tilde \theta I_{3,h}^k$)
 		    starting with the current mesh $h (= h^1_k)$		
 		    to obtain a regularization parameter $\beta_k$ and a possibly different discretization $h^2_k$
  		    such that \eqref{eq_IRGNM_inexNewton} holds; Therwith, also the corresponding 
  		    $\vhk = \vhzweik$, $\qhk = \qhzweik$ according to \eqref{eq_IRGNM_optprobinalg} are computed.  
   	   \ENDIF
   	   \STATE Set $h=h^2_k$.
   	   \STATE Evaluate the error estimator $\eta_1^k$ (cf. \eqref{eq_IRGNM_aao_I1234k},\eqref{eq_IRGNM_aao_I1234kdiscrete}).
   	   \STATE Set $h^3_k=h^2_k$.
 	   \WHILE {\eqref{eq_IRGNM_etacond1_algo} is violated} 
 		  \STATE Refine grid with respect to $\eta_1^k$ such that we obtain a finer discretization
 		  $h^3_k$.
                   \STATE Solve the the optimization problem \eqref{eq_IRGNM_optprobinalg}
      		  and evaluate $\eta_1^k$.
 	   \ENDWHILE
 	   \STATE Set $h=h^3_k$.
 	   \STATE Set $\qoldkplus = q_h^k$, $\uoldkplus = \uoldk+v_h^k$.
 	   \STATE Compute the adjoint state $\zhkplus( =z_{h^3_k}^{k+1})$ 
		    (see \eqref{eq_ggn_Lu}), evaluate 
  		     $\normklein{\zhkplus}_{W_h}$, set $\varrho_k=\max\{\varrho_{k-1}\,,\normklein{\zhkplus}_{W_h}\}$.
  		     and evaluate $I_{3,h}^{k+1}$.
 	   \STATE Set $h^1_{k+1} = h^3_k$ (i.e. use the current mesh as a starting mesh for the next iteration).  	   \STATE Set $k=k+1$. 
       \ENDWHILE   
     \end{algorithmic}
  \end{algorithm}

\begin{rem}\label{rem_IRGNM_ifstattwhile}
In practice, we replace the ``while''-loop on lines 15-19 of Algorithm \ref{alg_IRGNM_GGN} and in the Algorithm from \cite{GKV}
which only serve as refinement loops by an ``if''-condition in order to prevent over-refinement. 
Since we want to either refine or make a Gauss Newton step, lines 15-19 are replaced by
\begin{algorithm}.
  \begin{algorithmic}[1]
   \IF {\eqref{eq_IRGNM_etacond1_algo} is violated}
 		  \STATE Refine grid with respect to $\eta_1^k$ such that we obtain a finer discretization
 		  $h^3_k$
 		  \STATE $h=h^3_k$. 
  \ELSE
   	      \STATE Set $\qoldkplus = q_h^k$, $\uoldkplus = \uoldk+v_h^k$.
  \ENDIF
\end{algorithmic}
\end{algorithm}
\end{rem}

The structure of the loops is the same as in Algorithm 1 from \cite{KKVV13}, 
but here, we only have to solve linear PDEs (i.e. Step 6 in Algorithm 4 in \cite{KKVV13}
is replaced by 
``Solve linear PDE''), which justifies the drawback of one additional loop in comparison to
\cite{KKV10} (see also Algorithm 5 in \cite{KKVV13}.
This motivates the implementation and assumes
the gain of computation time for strongly nonlinear problems, which will be considered in terms of
numerical tests in Section \ref{sec_IRGNM_numericalresults}.

\section{Numerical Results}\label{sec_IRGNM_numericalresults}

For illustrating the performance of the proposed method according to Algorithm \ref{alg_IRGNM_GGN}, 
we apply it to the example PDE
 \[  
   \left\{
    \begin{aligned} 
    - \Delta u + \zeta u^3 & = q & \textnormal{in } \Omega\\
     u & = 0 & \textnormal{on } \partial\Omega 
    \end{aligned}
   \right.\,,
 \]
where we aim to identify the parameter $q\in Q=L^2(\Omega)$ from noisy measurements $g^\delta \in G$ 
of the state $u\in H^1_0(\Omega)$ 
in $\Omega = (0,1)^2 \subset \R^2$, where $\zeta>0$ is a given constant. As for the measurements we consider two cases:
\begin{itemize}
\item[(i)] via point functionals in $n_m$ uniformly distributed points $\xi_i$, $i=1,2,\dots,n_m$ and
perturbed by uniformly distributed random noise of some percentage $p>0$. Then the observation space is chosen as $G=\R^{n_m}$ and 
the observation operator is defined by $(C(v))_i = v(\xi_i)$ for $i=1,\dots, n_m$.
\item[(ii)] via $L^2$-projection. Then $G=L^2(\Omega)$, $C=\operatorname{id}$, and 
\[
  g^\delta = g+\delta \frac{r}{\norm{r}_{L^2(\Omega)}} = g + p\norm{g}_{L^2(\Omega)} \frac {r} {\norm{r}_{L^2(\Omega)}}\,,
\]
where $r$ denotes some uniformly distributed random noise and $p$ the percentage of perturbation. The exact state $u^\dagger$ 
is simulated on a very fine mesh with $1050625$ nodes and equally sized quadratic cells, and we denote the corresponding finite element space by $V_{h_L}$. 
In order to evaluate $\norm{C(u)-g^\delta}_{L^2(\Omega)} = \norm{u-g^\delta}_{L^2(\Omega)}$ on coarser meshes and the 
corresponding finite element spaces $V_{h_l}$ with $l=0,1,\dots,L$ during the optimization algorithm, $g^\delta$ has to
be transferred from $V_{h_L}$ to the current grid $V_{h_l}$. As usual in the finite element context, this is done by the $L^2$-projection
as the restriction operator.
\end{itemize}

We consider configurations with three different exact sources $q^\dagger$:
\begin{itemize}
\item[(a)]
A Gaussian distribution
\[
  q^\dagger = \frac {c} {2\pi \sigma^2} \exp\left(-\frac 12 \left(\left(\frac {sx-\mu} {\sigma}\right)^2 +\left(\frac{sy-\mu}{\sigma}\right)^2 \right)\right) 
\]
with $c=10$, $\mu = 0.5$, $\sigma=0.1$, and $s=2$.
\item[(b)]	Two Gaussian distributions added up to one distribution
\[
  q^\dagger = q_1 + q_2,
\]
where
\[
  q_1 =  \frac {c_1} {2 \pi \sigma^2} \exp \left(-\frac 1 2 \left( \left(\frac {s_1x-\mu} {\sigma} \right)^2 + 
  \left(\frac {s_1y-\mu} {\sigma} \right)^2 \right) \right), 
\]
\[
  q_2 = \frac {c_2} {2 \pi \sigma^2} \exp \left(-\frac 1 2 \left( \left(\frac {s_2x-\mu} {\sigma} \right)^2 + 
  \left(\frac {s_2y-\mu} {\sigma} \right)^2 \right)\right)
\]
with $\sigma = 0.1$, $\mu = 0.5$, $s_1= 2$, $s_2 = 0.8$, $c_1=1$, and $c_2=1$.
\item[(c)]The step function
\[
q^\dagger = \left\{
\begin{array}{cll} 
 0 	 & \mbox{for } x \ge \frac 1 2  \\
 1 	 & \mbox{for } x < \frac 1 2\,. \\ 
\end{array}\right.
\]
\end{itemize}

The concrete choice of the parameters for the numerical tests is as follows: 
$c_{tc} = 10^{-7}$, $\ottheta=0.4999$, $\uttheta = 0.2$, $\tau = 5$, $\tau_\beta = 1.66$, 
$\tilde \tau_\beta = 1$, ($c_2=0.9999$, $c_3 = 0.0001$). The coarsest (starting mesh) consists 
of 25 nodes and 16 equally sized squares, the inital values for the control and the state 
are $q_0=0$ and $u_0=0$ and we start with a regularization parameter $\beta=10$.

Considering the numerical tests, we are mainly interested in saving computation time compared to the Algorithm from \cite{KKV10}, where the inexact Newton
method for the determination of the regularization parameter $\beta$ is applied directly to the nonlinear problem, instead 
of the linearized subproblems \eqref{eq_IRGNM_optprobinalg}. That is why besides the numerical results for the Generalized Gauss-Newton (GGN)
method presented in section \ref{subsec_IRGNM_sqp}, we also present the results from the ``Nonlinear Tikhonov'' (NT) Algorithm from \cite{KKV10}.

The choice of the parameters for (NT) is the following: $\tilde \tau =0.1$, $\underline{\underline\tau} = 3.1$, 
$\tau = 4$, $\overline{\overline \tau} = 5$, $c_{tc}=10^{-7}$, $c_1=0.9$, $c_2=0.4$. This setting implies that both algorithms (NT) and (GGN) are 
stopped, if the concerning quantities of interest fall below the same bound ($\overline{\overline \tau}^2\delta^2$ for (NT) and
$\tau^2\delta^2$ for (GGN)).

\begin{figure}[htbp]
\hspace*{-2cm}
  \begin{minipage}[b]{4.5cm}
    \includegraphics[scale=0.3]
    {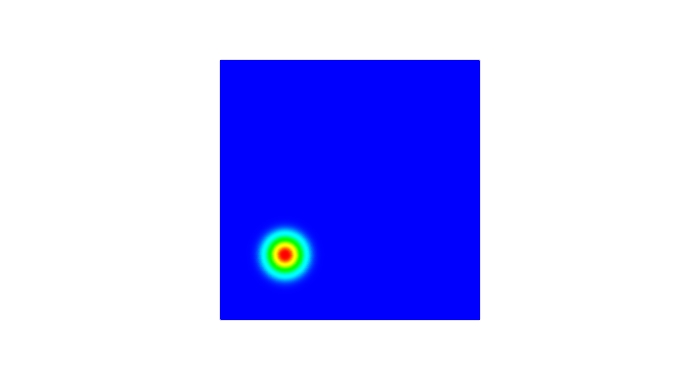}
  \end{minipage}
  \hspace{5mm}
  \begin{minipage}[b]{4.5cm}
    \includegraphics[scale=0.3]
    {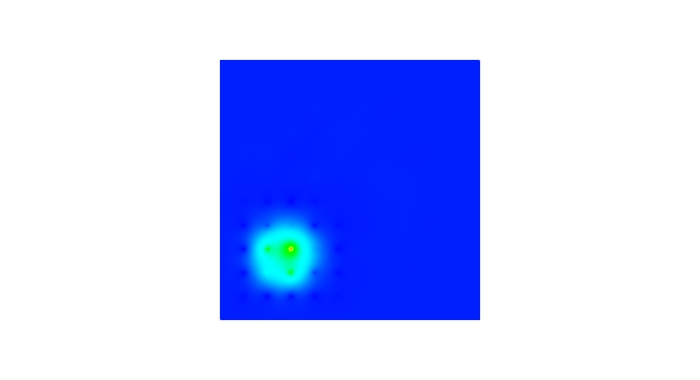}
  \end{minipage}
  \hspace{5mm}
  \begin{minipage}[b]{4.5cm}
    \includegraphics[scale=0.3] 
    {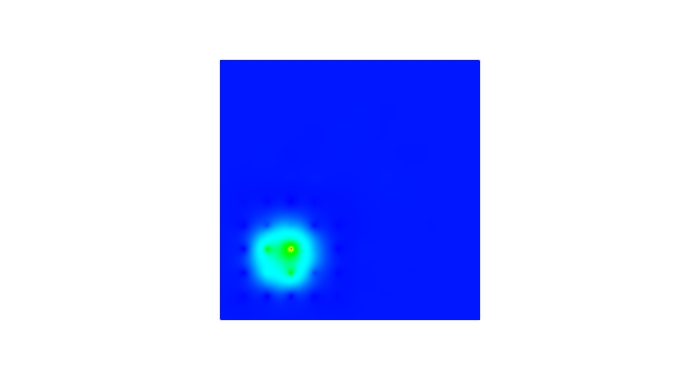}
  \end{minipage}
  \caption{FLTR: exact control $q^\dagger$, reconstructed control by NT, reconstructed control by GGN
    for example (a) (i) with $\zeta=100$, $1\%$ noise}
  \label{fig:GaussqPointzeta=100}
\end{figure}

\begin{figure}[htbp]
\hspace*{-2cm}
  \begin{minipage}[b]{4.5cm}
    \includegraphics[scale=0.3]
    {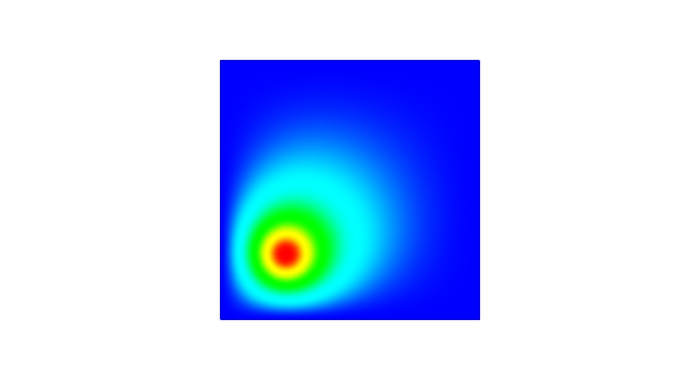}
  \end{minipage}
  \hspace{5mm}
  \begin{minipage}[b]{4.5cm}
    \includegraphics[scale=0.3]
    {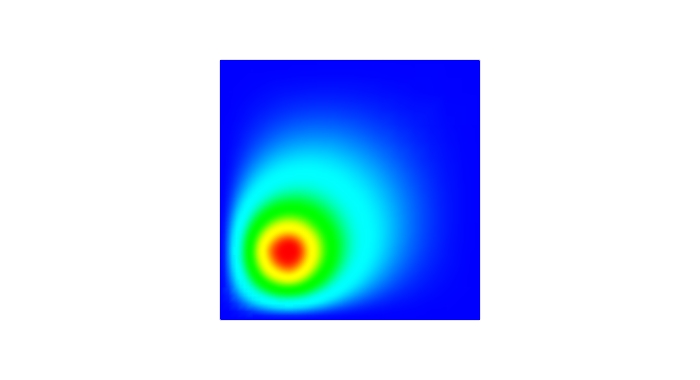}
  \end{minipage}
  \hspace{5mm}
  \begin{minipage}[b]{4.5cm}
    \includegraphics[scale=0.3] 
    {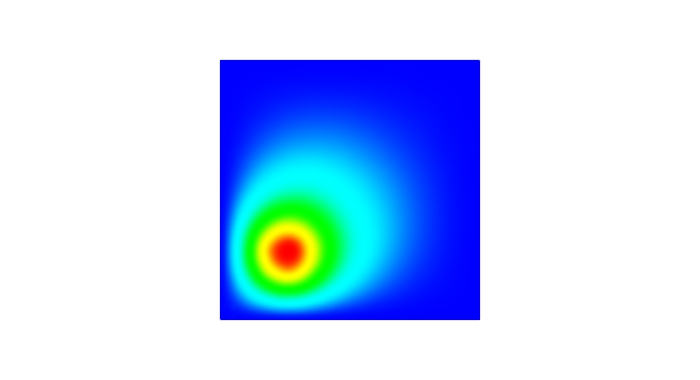}
  \end{minipage}
  \caption{FLTR: exact state $u^\dagger$, reconstructed state by NT, reconstructed state by GGN
    for example (a) (i) with $\zeta=100$, $1\%$ noise}
  \label{fig:GaussuPointzeta=100}
\end{figure}


\begin{figure}[htbp]  
  \hspace{53mm}
  \begin{minipage}[b]{4cm}
    \includegraphics[scale=0.16]
    {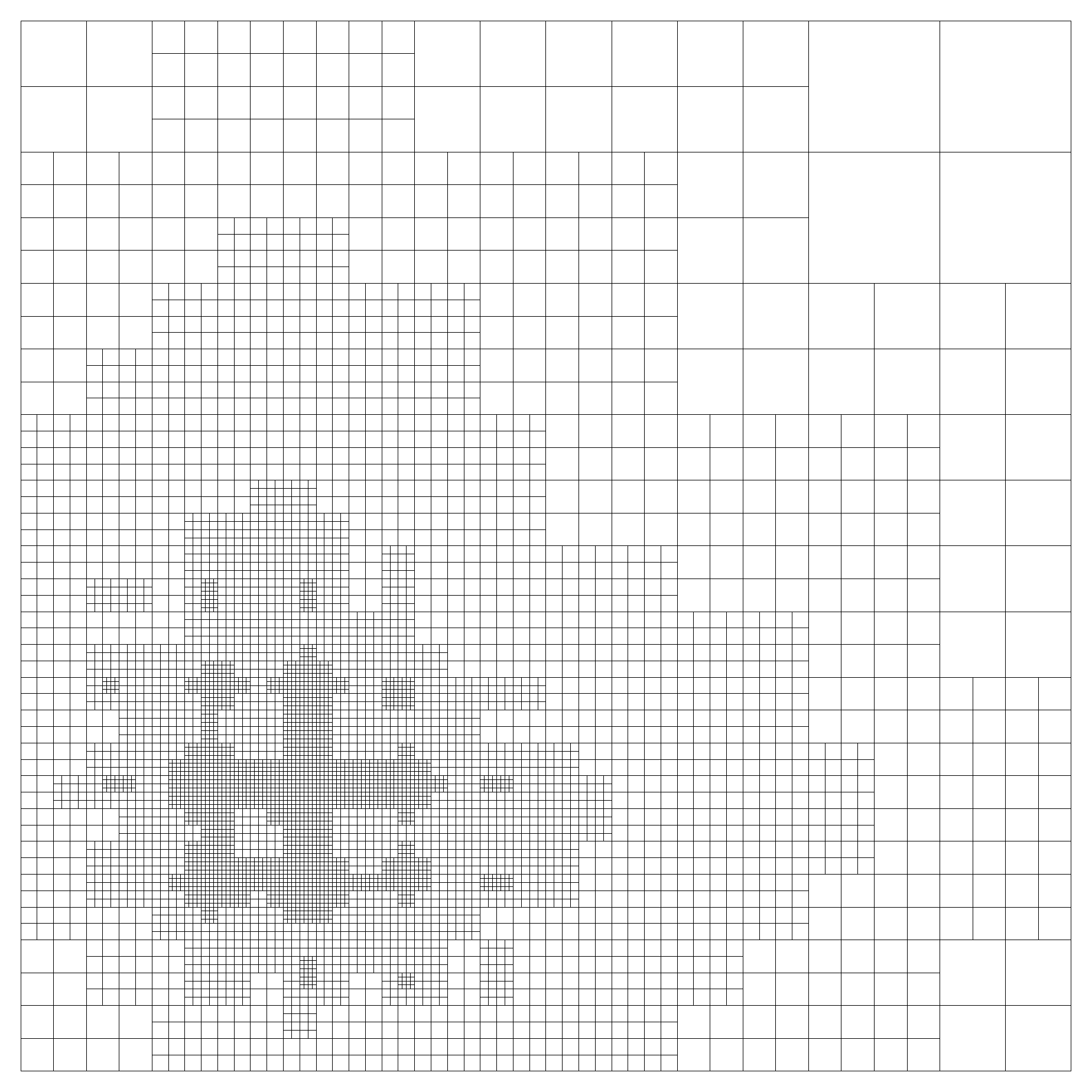}
  \end{minipage}
  \hspace{10mm}
  \begin{minipage}[b]{4cm}
    \includegraphics[scale=0.16] 
    {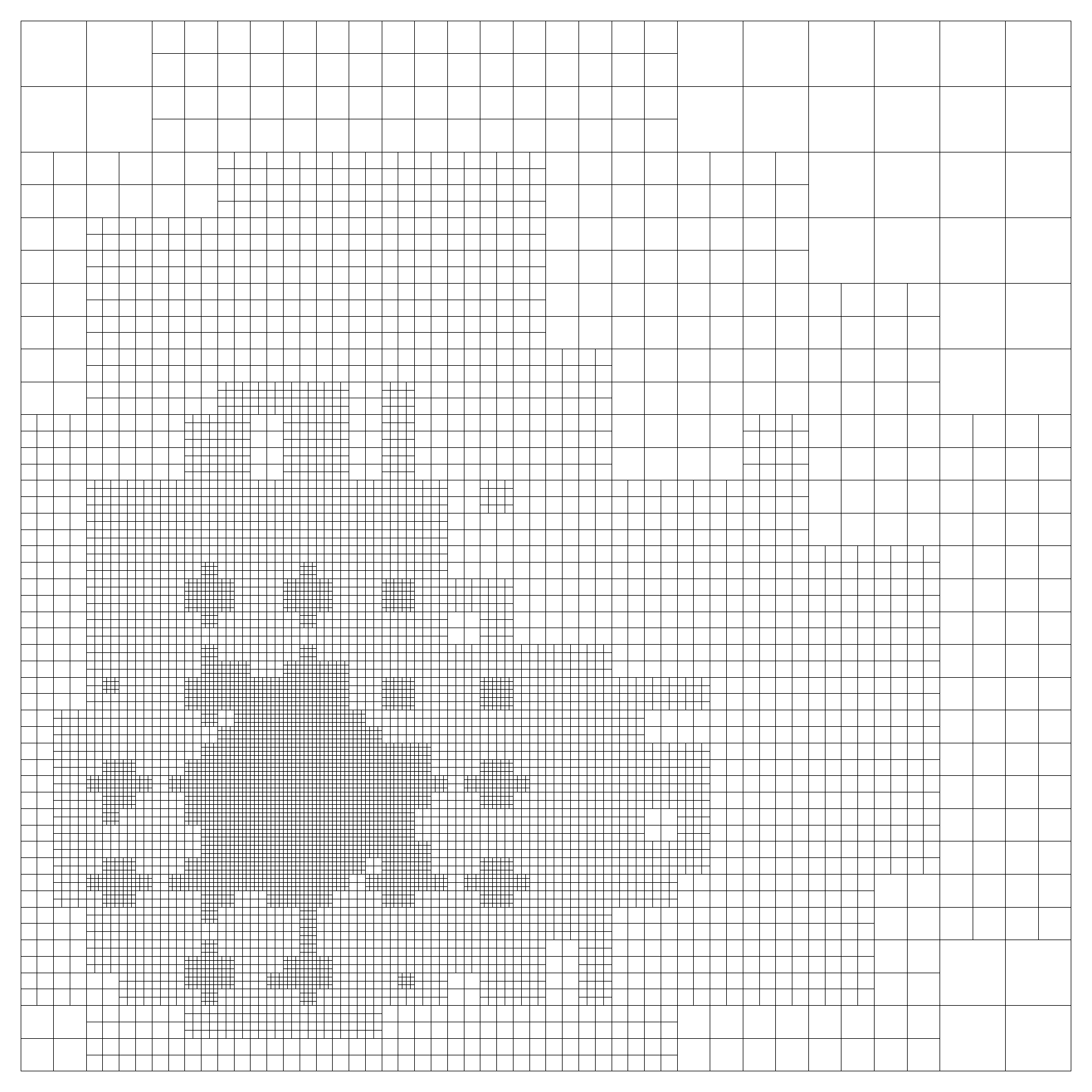}
  \end{minipage}
  \caption{FLTR: adaptively refined mesh by NT , adaptively refined mesh by     GGN
    for example (a) (i) with $\zeta=100$, $1\%$ noise}
  \label{fig:GaussmeshPointzeta=100}
\end{figure}

The figures \ref{fig:GaussqPointzeta=100} and \ref{fig:GaussuPointzeta=100} show the exact source distribution $q^\dagger$ and the corresponding simulated state 
$u^\dagger$, as well as the reconstructions of the control and the state obtained by Algorithm \ref{alg_IRGNM_GGN} (GGN), 
as well as the ones obtained by the algorithm from \cite{KKV10} (NT) for the example (a)(i) with $\zeta=100$ and $1\%$ noise. In Figure \ref{fig:GaussmeshPointzeta=100} we see the very fine mesh for simulating the data, 
the adaptively refined mesh obtained by (NT), and the adaptively refined mesh obtained by (GGN). 

In table \ref{tab:zetas} we present the respective results for different choices of $\zeta$ (first column). In the second and fifth column one can see the relative 
control error $\frac {\norm{q_h^{k_*} - q^\dagger}_Q^2} {\norm{q^\dagger}_Q}$, in the third and sixth column is the number of nodes in the adaptively refined mesh, 
and in the forth and seventh column one can see the regularization parameter obtained by (GGN) and (NT) respectively. The eigthth column shows the 
gain of computation time using (GGN) instead of (NT). The higher the factor $\zeta$ is, the more computation time we save with (GGN). This is probably due 
to the higher number of iterations needed for ``more nonlinear'' problems. Already for the choice $\zeta=100$ replacing the nonlinear PDEs by linear ones 
(getting an additional loop in return cf. subsection \ref{sec_IRGNM_aao_alg}) seems to pay off. In that case (GGN) refines more than (NT)
(see Figure \ref{fig:GaussmeshPointzeta=100}), but it is still faster (see table \ref{tab:zetas}). For higher $\zeta=500$ and $\zeta=1000$ (GGN)
is even much faster than (NT), because in addition to the cheaper linear PDEs, it also refines less. At the same time, the relative control error 
is about the same as with (NT).

\begin{table}
  \caption{Algorithm \ref{alg_IRGNM_GGN} (GGN) versus the algorithm from \cite{KKV10} (NT) 
  for Example (a)(i) for different choices of $\zeta$ with $1\%$ noise. CTR: Computation time reduction using (GGN) in comparison to (NT)}
  \label{tab:zetas}
  \centering
  \begin{tabular}{@{}rrrrrrrr@{}}
    \toprule
    \multicolumn{1}{@{}c}{$\zeta$}
    &\multicolumn{3}{c@{}}{NT}
    &\multicolumn{3}{c@{}}{GGN}
    &\multicolumn{1}{@{}c}{CTR} \\
    \cmidrule(r){1-1}\cmidrule(lr){2-4}\cmidrule(lr){5-7}\cmidrule(l){8-8}
    &  error & $ \beta$ & \# nodes &   error & $ \beta$ & \# nodes & \\   
    \cmidrule(r){2-2}\cmidrule(lr){3-3}\cmidrule(lr){4-4}\cmidrule(lr){5-5}\cmidrule(lr){6-6}\cmidrule(l){7-7}
    1 & 0.418  & 2985 & 2499 & 0.412 & 4600 & 3873 & -65\% \\
    10 & 0.417  & 3194 & 2473 & 0.411 & 4918 & 3965 & -59\% \\
    100 & 0.408  & 5014 & 6653 & 0.417 & 6773 & 9813 & 39\% \\
    500 & 0.418  & 9421 & 11851 & 0.404 & 13756 & 821 & 97\% \\
    1000 & 0.439  & 11486 & 44391 & 0.426 & 16355 & 793 & 99\% \\  
    \bottomrule
  \end{tabular}
\end{table}

In table \ref{tab:differentnoise} the reader can see the results for the same example with $\zeta=100$ for different noise levels using (GGN). 
The numerical results confirm what we would expect: the larger the noise, the larger the error, the stronger the 
regularization, the coarser the discretization. 

\begin{table}[t]
  \centering
  \caption{Example (a)(i) for different noise levels with $\zeta=100$}\label{tab:differentnoise}
  {\footnotesize
  \begin{tabular}{@{}lrrr@{}}
    \toprule
    \multicolumn{1}{@{}c}{noise}&\multicolumn{1}{@{}c}{error}&\multicolumn{1}{@{}c}{$ \beta$}&\multicolumn{1}{@{}c}{\# nodes} \\
    \cmidrule(r){1-1}\cmidrule(lr){2-2}\cmidrule(lr){3-3}\cmidrule(l){4-4}
    0.5\%   & 0.385  & 12260 & 9867 	    \\
    1\%   & 0.417  &   6773 & 9813    \\
    2\%   & 0.553  &    1687 & 413   \\
    4\%   &  0.700 &  599 & 137	      \\    
    8\%	  & 0.937 &  42 & 137\\
    \bottomrule
  \end{tabular}}
\end{table}

\begin{figure}[htbp]
\centering
    \includegraphics[scale=0.7]
    {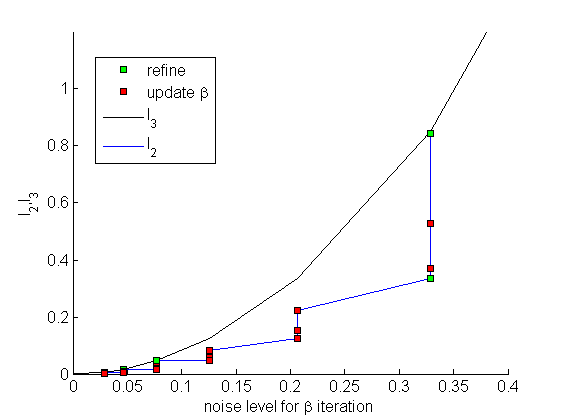}
   \label{fig:ggn}
   \caption{Behavior of (GGN) for example (a)(i) with $\zeta=100$ and $1\%$ noise}
\end{figure}

Taking a look at Figure \ref{fig:ggn} the reader can track the behavior of Algorithm \ref{alg_IRGNM_GGN} (GGN)
for the considered example (a)(i) with $\zeta=100$ and $1\%$ noise. The algorithm goes from right to left 
  in Figure \ref{fig:ggn}, where the quantities of interest $I_2$ and $I_3$ (or rather their discrete counterparts $I_{3,h}$ 
  and $I_{2,h}$) are rather large. The noise level for the inner iteration $\tilde \theta I_{3,h}$ is about $0.52$ in the beginning.
  For this noise 
  level the stopping criterion for the $\beta$-algorithm (step 10,11 in Algorithm \ref{alg_IRGNM_GGN})
  is already fulfilled, such that only one Gauss-Newton step is made without refining or updating $\beta$.
  This decreases the noise level $\tilde \theta I_{3,h}$ to about $0.33$. Then the $\beta$-algorithm comes into play, 
  with one refinement step, two $\beta$-steps and again one refinement step, which in total reduces $I_2$
  from $0.90$ to $0.33$, with which the $\beta$-algorithm terminates. The subsequent run of the $\beta$-algorithm
  consists only of three $\beta$-enlargement steps and finally after $7$ Gauss-Newton iterations, both quantites 
  of interest $I_2$ and $I_3$ fulfill the required smallness conditions such that the whole Gauss-Newton Algorithm terminates.

Due to the observation above concerning the nonlinearity of the PDE, 
we restrict our considerations to the case $\zeta=1000$ for the rest of this section. 
The figures \ref{fig:GaussqL2zeta=1000}, \ref{fig:GaussuL2zeta=1000}, and \ref{fig:GaussmeshL2zeta=1000} 
again show the results for example (a) with $1\%$ noise, but for the case (ii), i.e. via $L^2$-projection.

\begin{figure}[htbp]
\hspace*{-2cm}
  \begin{minipage}[b]{4.5cm}
    \includegraphics[scale=0.3]
    {GaussEcke_exactControl.jpg}
  \end{minipage}
  \hspace{5mm}
  \begin{minipage}[b]{4.5cm}
    \includegraphics[scale=0.3]
    {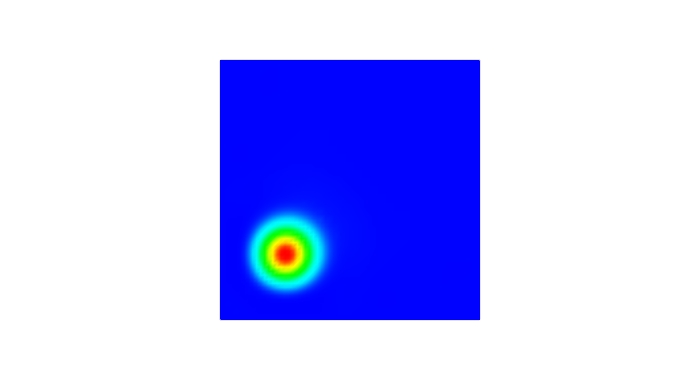}
  \end{minipage}
  \hspace{5mm}
  \begin{minipage}[b]{4.5cm}
    \includegraphics[scale=0.3] 
    {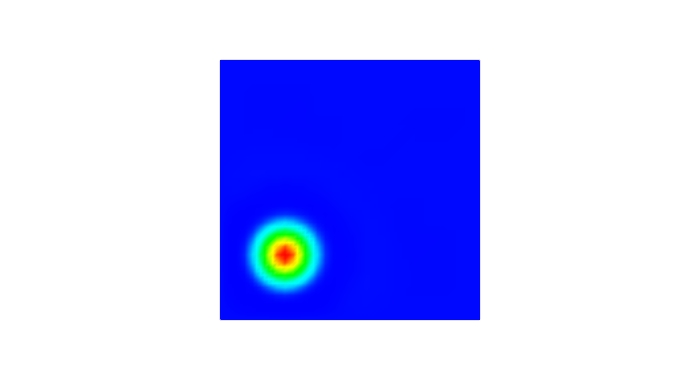}
  \end{minipage}
  \caption{FLTR: exact control $q^\dagger$, reconstructed control by NT, reconstructed control by GGN
    for example (a) (ii) with $\zeta=1000$, $1\%$ noise}
  \label{fig:GaussqL2zeta=1000}
\end{figure}

\begin{figure}[htbp]
\hspace*{-2cm}
  \begin{minipage}[b]{4.5cm}
    \includegraphics[scale=0.3]
    {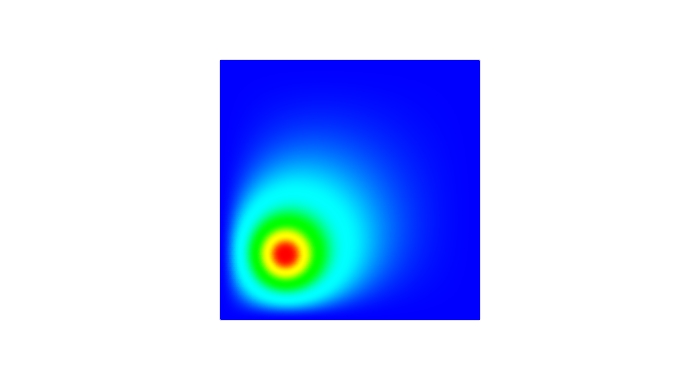}
  \end{minipage}
  \hspace{5mm}
  \begin{minipage}[b]{4.5cm}
    \includegraphics[scale=0.3]
    {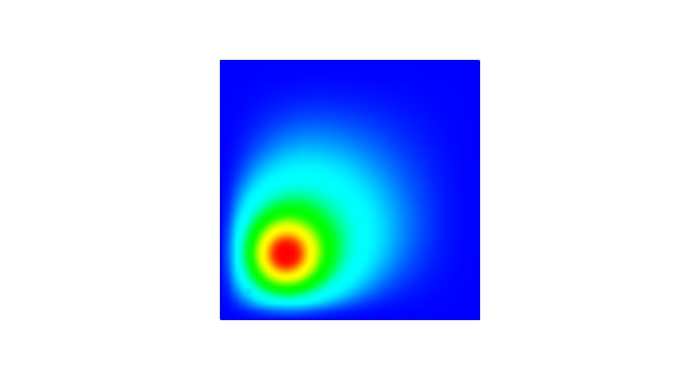}
  \end{minipage}
  \hspace{5mm}
  \begin{minipage}[b]{4.5cm}
    \includegraphics[scale=0.3] 
    {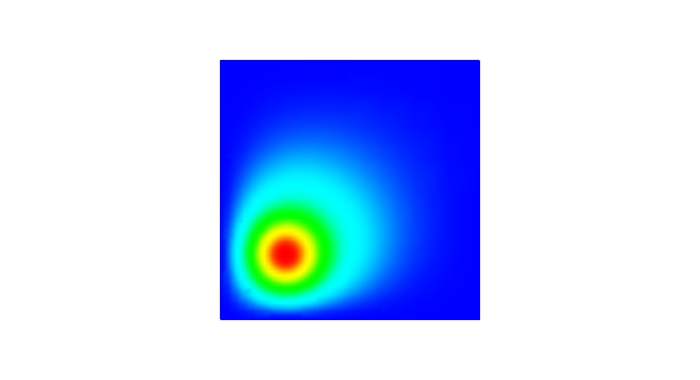}
  \end{minipage}
  \caption{FLTR: exact state $u^\dagger$, reconstructed state by NT, reconstructed state by GGN
    for example (a) (ii) with $\zeta=1000$, $1\%$ noise}
 \label{fig:GaussuL2zeta=1000}
\end{figure}
%

\begin{figure}[htbp]  
\hspace{53mm}
  \begin{minipage}[b]{4cm}
    \includegraphics[scale=0.16]
    {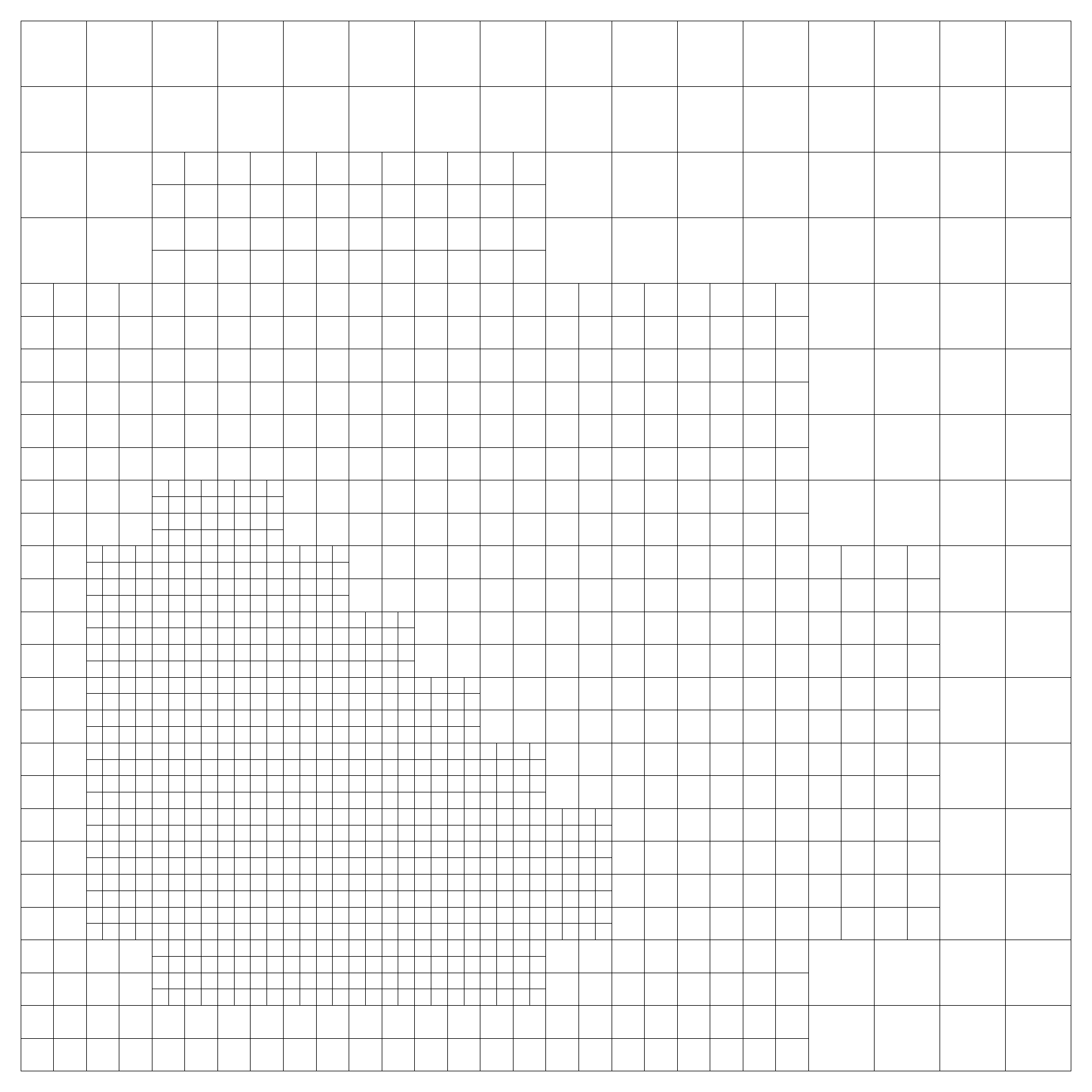}
  \end{minipage}
  \hspace{10mm}
  \begin{minipage}[b]{4cm}
    \includegraphics[scale=0.16] 
    {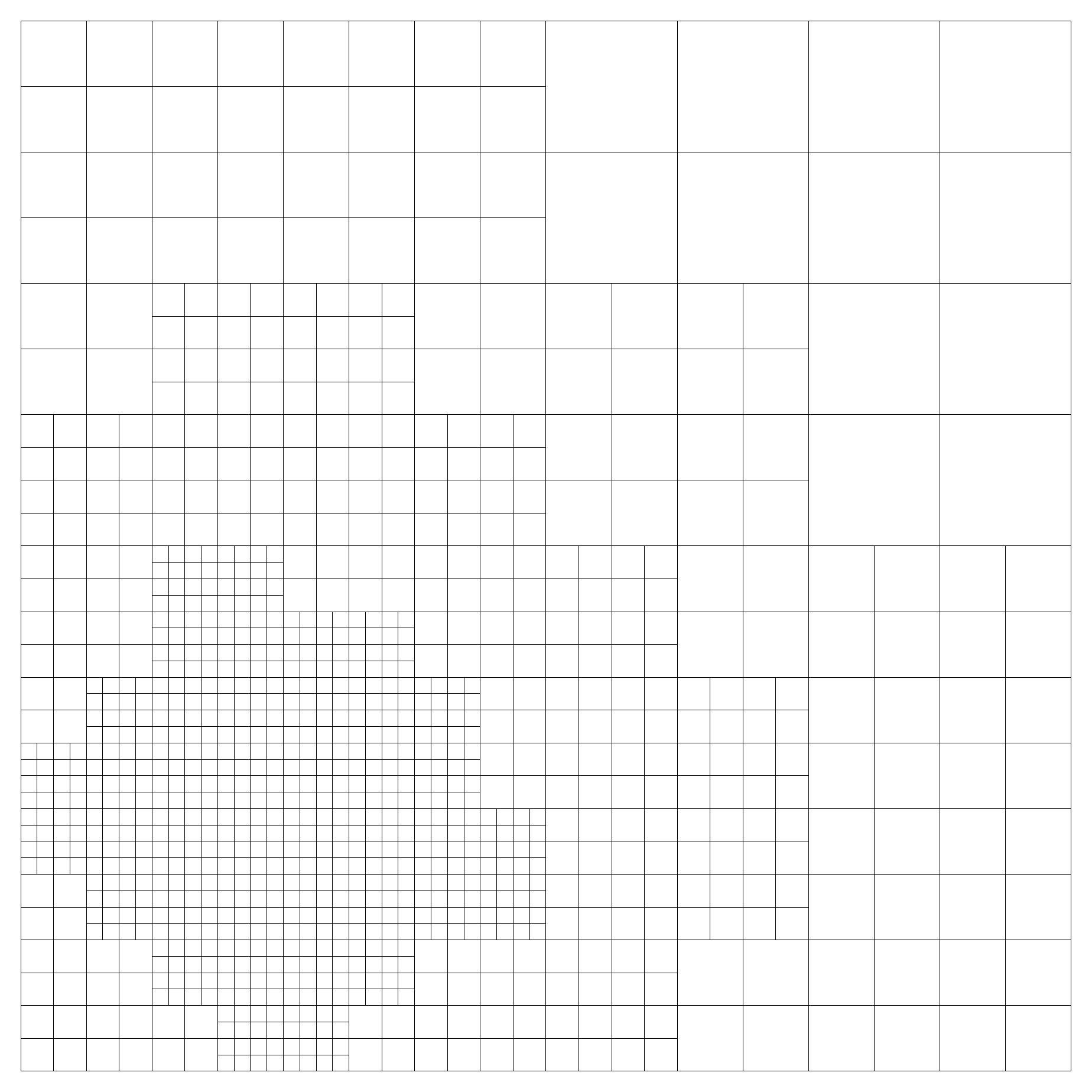}
  \end{minipage}
  \caption{FLTR: adaptively refined mesh by NT , adaptively refined mesh by     GGN
    for example (a) (ii) with $\zeta=1000$, $1\%$ noise}
  \label{fig:GaussmeshL2zeta=1000}
\end{figure}

(GGN) yields a regularization parameter $\beta = 4400806$, a discretization with $1125$ nodes and a relative control error of $0.268$. 
(NT) leads to a much larger error of $1.472$, a finer discretization with $1405$ nodes and a much larger regularization parameter $\beta = 60875207$.
Although (GGN) refines only a little less than (NT), (GGN) is much faster than (NT), namely $81\%$. 
Compared to the point measurement evaluation, the $L^2$-projection causes smoother solutions, which seem to reconstruct the exact data better, 
but at the same time this is probably the less realistic case with respect to real applications.

In the figures \ref{fig:HuetchenqPointzeta=1000} and \ref{fig:zweiHuetchenmeshPointzeta=1000}, we can see the results using (GGN) and (NT) 
for a different source, namely example (b) with point measurements (i) and again $\zeta=1000$ and $1\%$ noise. Since we are interested in idenfying the parameter 
$q$, we take a pass on presenting the reconstructed states and only show the reconstructed controls, as well as the adaptively refined meshes. 

\begin{figure}[htbp]
\hspace*{-2cm}
  \begin{minipage}[b]{4.5cm}
    \includegraphics[scale=0.3]
    {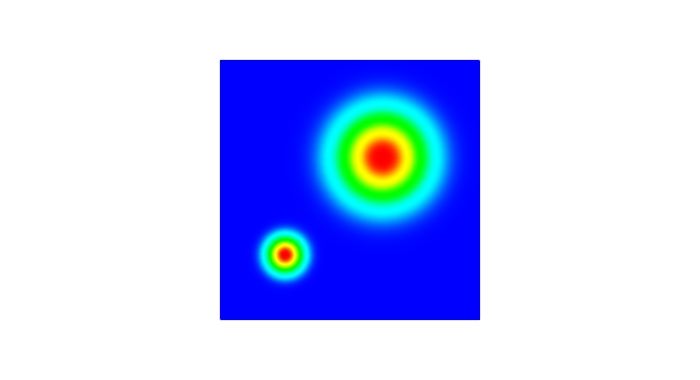}
  \end{minipage}
  \hspace{5mm}
  \begin{minipage}[b]{4.5cm}
    \includegraphics[scale=0.3]
    {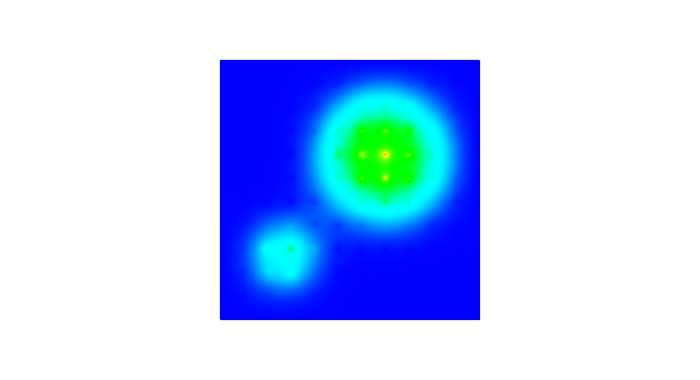}
  \end{minipage}
  \hspace{5mm}
  \begin{minipage}[b]{4.5cm}
    \includegraphics[scale=0.3] 
    {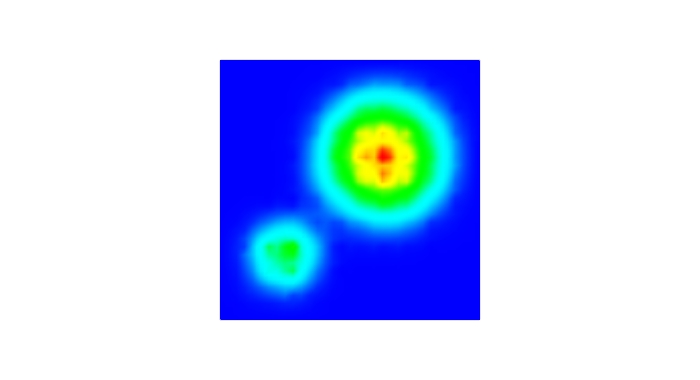}
  \end{minipage}
  \caption{FLTR: exact control $q^\dagger$, reconstructed control by NT, reconstructed control by GGN
    for example (b) (i) with $\zeta=1000$, $1\%$ noise}
  \label{fig:HuetchenqPointzeta=1000}
\end{figure}


\begin{figure}[htbp]  
\hspace{53mm}
   \begin{minipage}[b]{4cm}
    \includegraphics[scale=0.16]
    {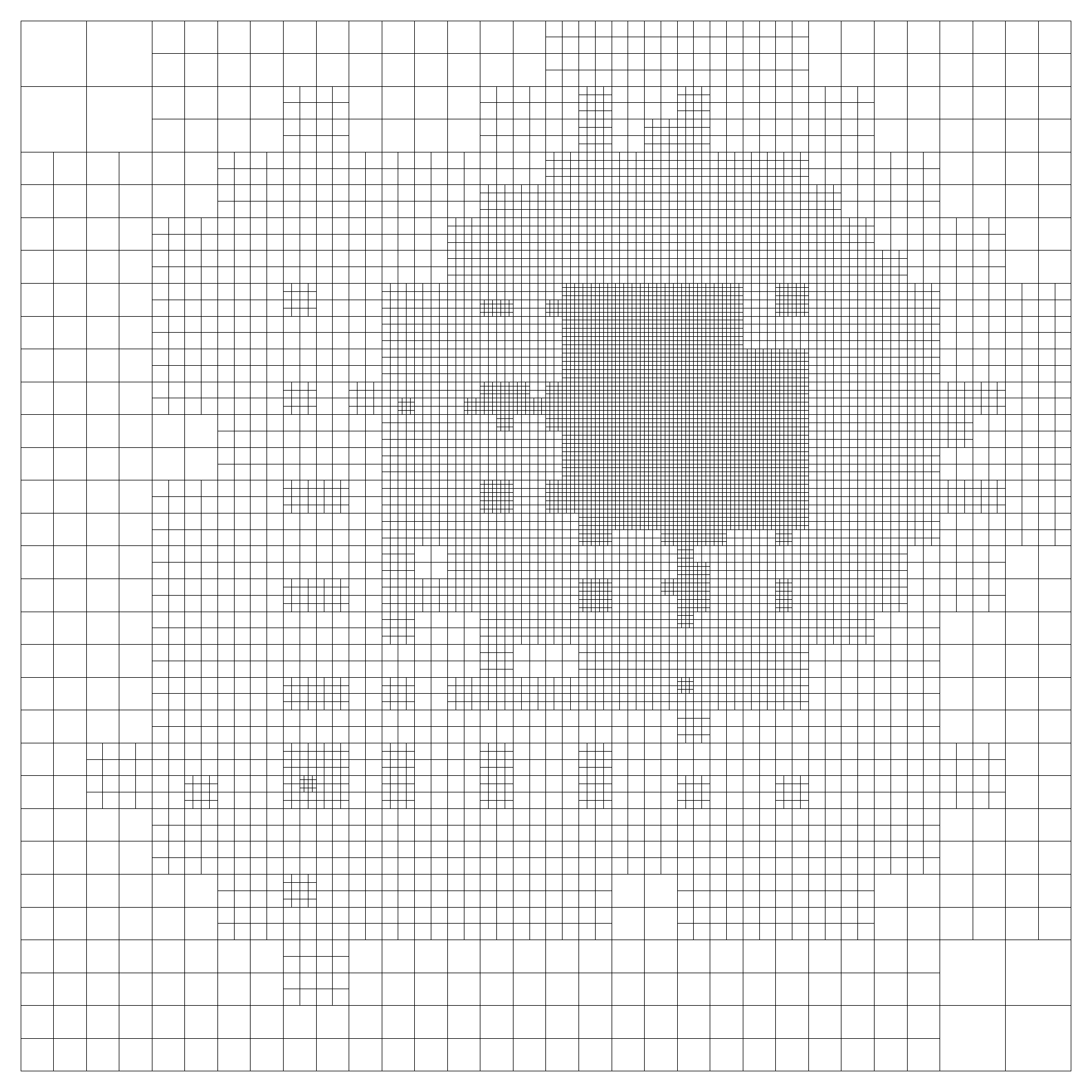}
  \end{minipage}
  \hspace{10mm}
  \begin{minipage}[b]{4cm}
    \includegraphics[scale=0.16] 
    {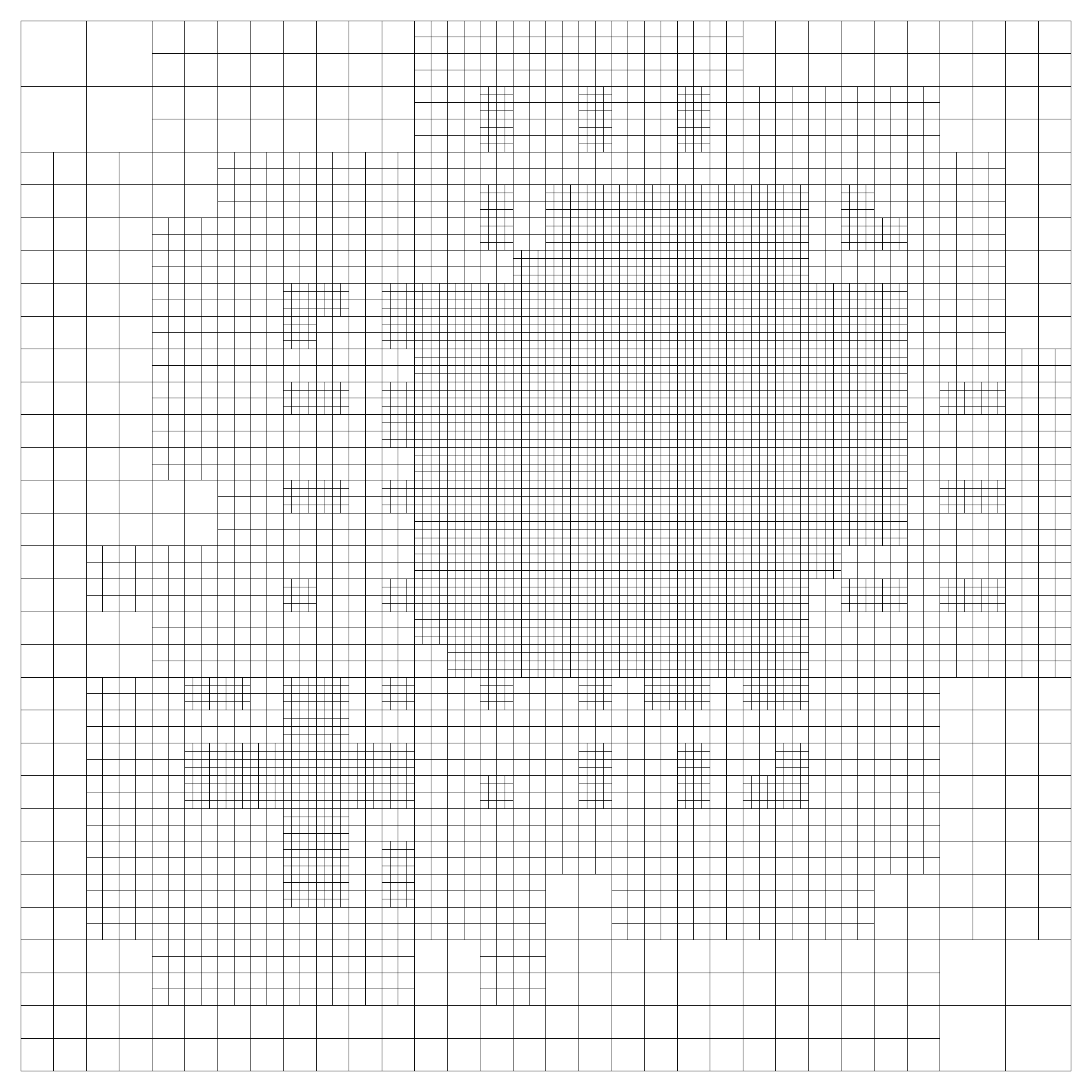}
  \end{minipage}
  \caption{FLTR: exact (very fine) mesh, adaptively refined mesh by NT , adaptively refined mesh by     GGN
    for example (b) (i) with $\zeta=1000$, $1\%$ noise}
  \label{fig:zweiHuetchenmeshPointzeta=1000}
\end{figure}

(GGN) stops with a regularization parameter $\beta=1616$, a mesh with $6697$ nodes, and a reconstruction yielding a relative error of $0.247$, 
whereas (NT) terminates with $\beta = 539$, $10063$ nodes and a larger error of $0.366$. 
Due to the much coarser discretization obtained by (GGN), it is not surprising, that we save about $26\%$ of computation time in this case.

The corresponding results for the source (c) are shown in Figure \ref{fig:SprungPointzeta=1000} and Figure \ref{fig:SprungmeshPointzeta=1000}.

\begin{figure}[htbp]
\hspace*{-2cm}
  \begin{minipage}[b]{4.5cm}
    \includegraphics[scale=0.3]
    {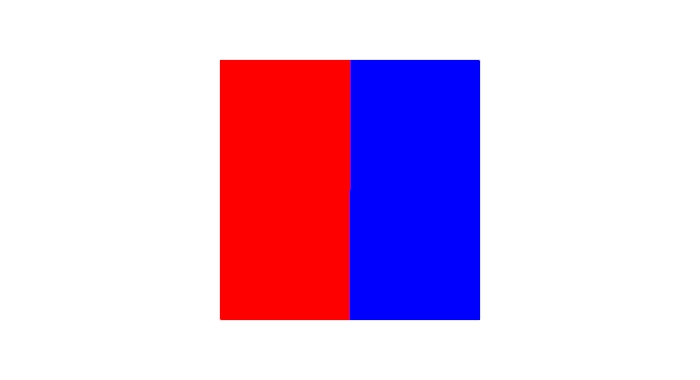}
  \end{minipage}
  \hspace{5mm}
  \begin{minipage}[b]{4.5cm}
    \includegraphics[scale=0.3]
    {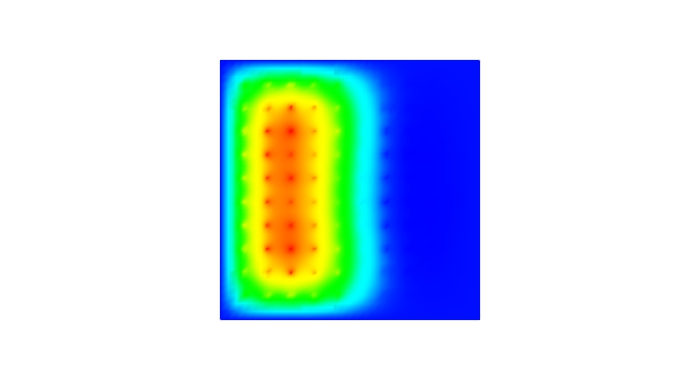}
  \end{minipage}
  \hspace{5mm}
  \begin{minipage}[b]{4.5cm}
    \includegraphics[scale=0.3] 
    {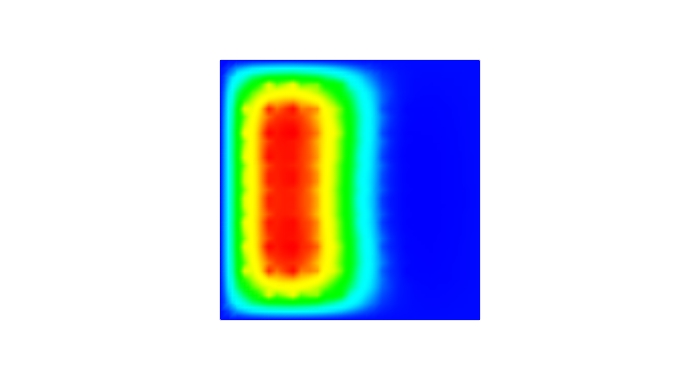}
  \end{minipage}
  \caption{FLTR: exact control $q^\dagger$, reconstructed control by NT, reconstructed control by GGN
    for example (c) (i) with $\zeta=1000$, $1\%$ noise}
  \label{fig:SprungPointzeta=1000}
\end{figure}


\begin{figure}[htbp]  
\hspace{53mm}
  \begin{minipage}[b]{4cm}
    \includegraphics[scale=0.16]
    {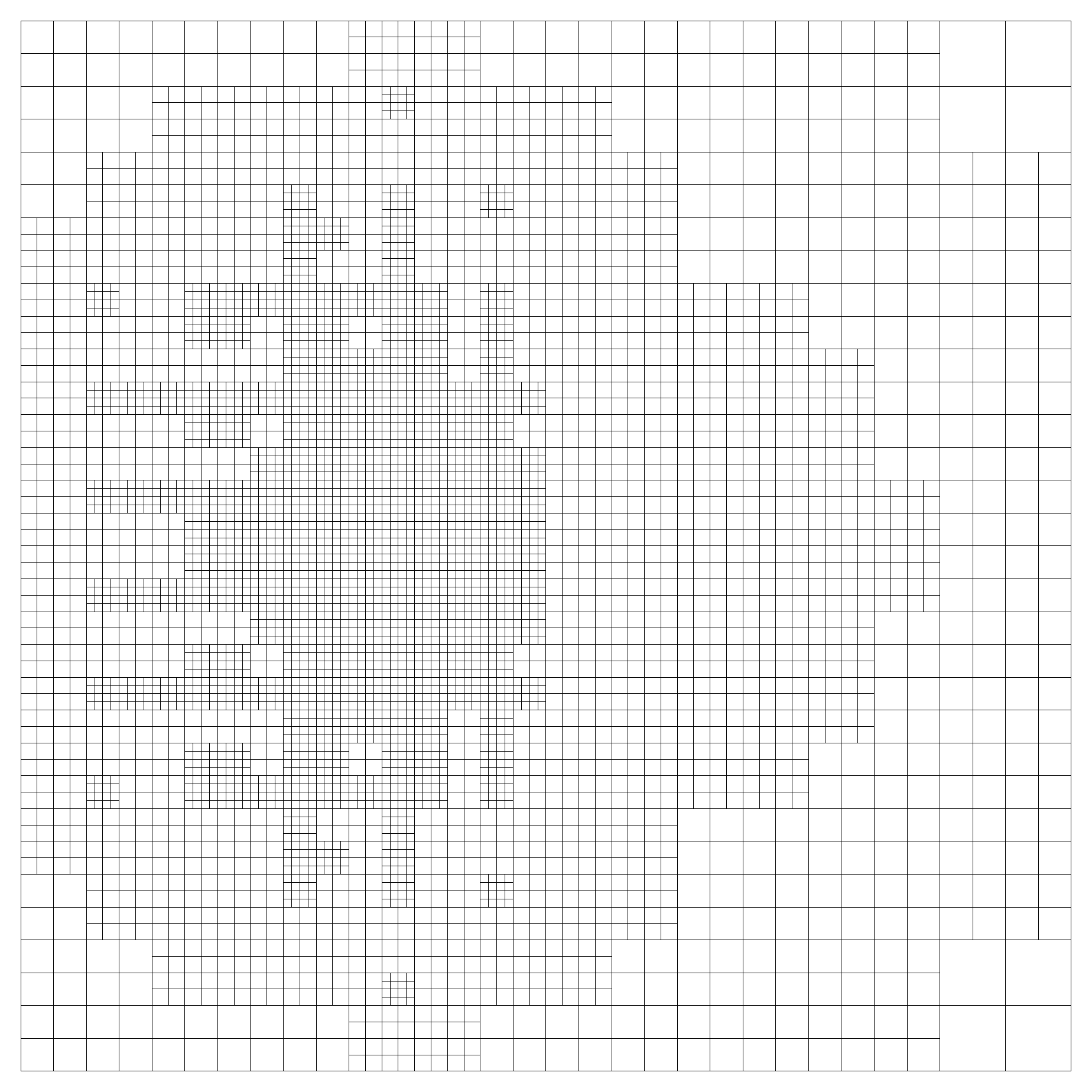}
  \end{minipage}
  \hspace{10mm}
  \begin{minipage}[b]{4cm}
    \includegraphics[scale=0.16] 
    {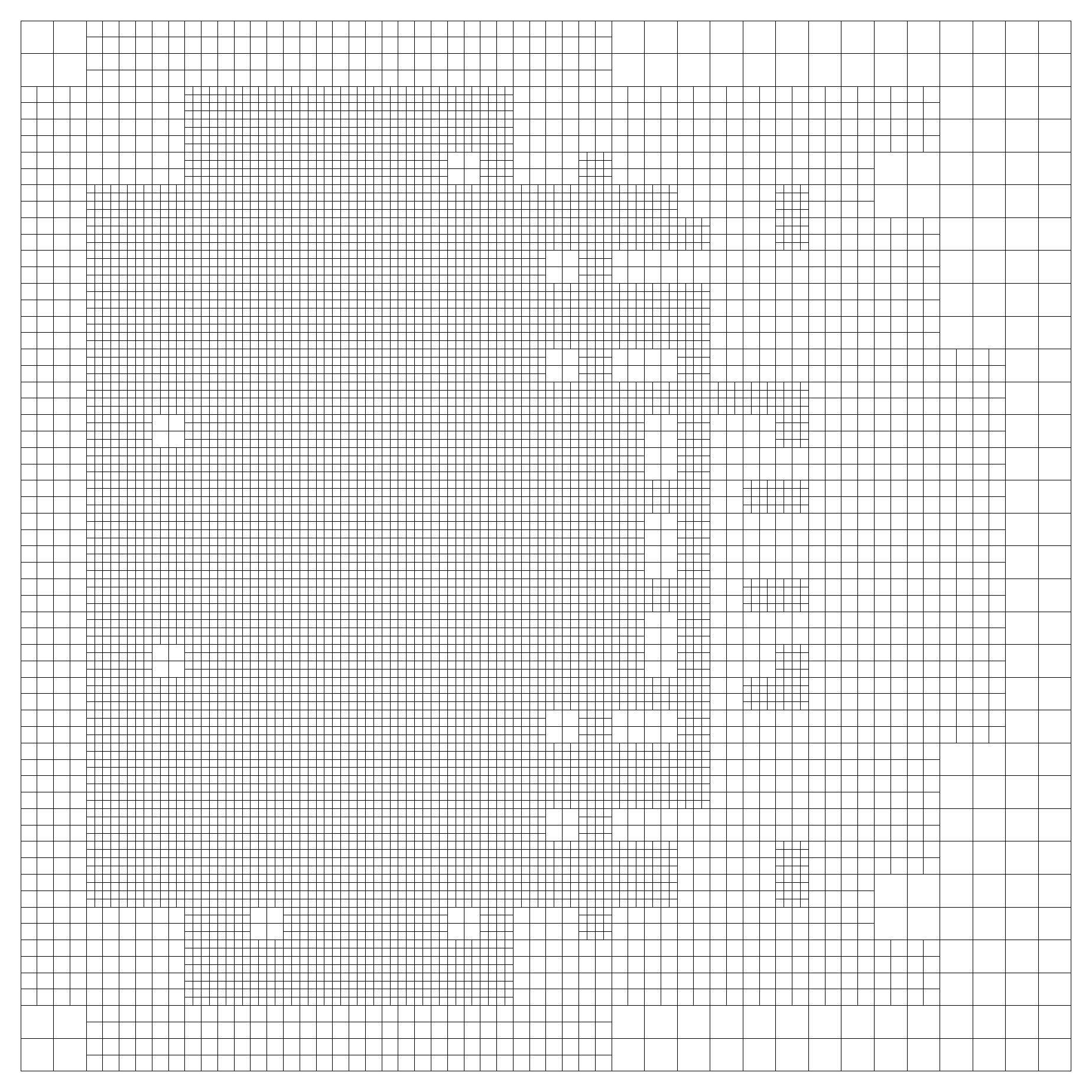}
  \end{minipage}
  \caption{FLTR: exact (very fine) mesh, adaptively refined mesh by NT , adaptively refined mesh by     GGN
    for example (c) (i) with $\zeta=1000$, $1\%$ noise}
  \label{fig:SprungmeshPointzeta=1000}
\end{figure}

Using (GGN) we obtain a regularization parameter $\beta=379$, a discretization with $9565$ nodes and a relative control error of $0.433$, while 
(NT) yields $\beta=24$, $5367$ nodes and an error of $0.615$. Also for this configuration (GGN) is faster than (NT), if only $9.4\%$. 
To put this in perspective, we would like to mention that the step function (c) is a very challenging example, since 
the intial guess $q_0=0$ and the source $q^\dagger$ have different values on the boundary. 
Moreover, for piecewise constant functions total variation regularization is known to yield much better results than $L^2$ regularization.

\section{Conclusions and Remarks}\label{sec_conclusions}
In this paper we consider all-at-once formulations of the iteratively regularized 
Gauss-Newton method and their adaptive discretizations using a posteriori error estimators. 
This allows us to consider only the linearized PDE (instead of the full potentially nonlinear one) 
as a constraint in each Newton step, which safes computational effort. Alternatively, in a least squares approach, 
the measurement equation and the PDE are treated simultaneously via unconstrained minimization of the squared 
residual. In both cases we show convergence and convergence rates which we carry over to the discretized setting 
by controlling precision only in four real valued quantities per Newton step. The choices of the regularization 
parameters in each Newton step and of the overall stopping index are done a posteriori, via a discrepancy type 
principle. From the numerical tests we have seen, that the presented method yields reasonable reconstructions and
can even lead to a large reduction of computation time compared to similar non-iterative methods. 

\section{Acknowledgments}\label{sec_acknowledgments}

The authors would like to thank the Federal Ministry of Education and Research (BMBF) for financial
support within the grant 05M2013 ``ExtremSimOpt: Modeling, Simulation and Optimization of Fluids in Extreme Conditions'',
as well as the German Science Foundation (DFG) for their support within the grant KA 1778/5-1 and VE 368/2-1 ``Adaptive Discretization
Methods for the Regularization of Inverse Problems''.

\section*{Appendix}
For proving convergence rates of the iterates according to \eqref{eq_IRGNM_aao}, we consider source conditions of the form 
\beq\label{source_aao}
\exists (s,v) \in Q\times V \mbox{ s.t. }(q^\dag -q_0,u^\dag-u_0)
 = \kappa(\mathbf{F}'(q^\dag,u^\dag)^*\mathbf{F}'(q^\dag,u^\dag)) (s,v)\,.
\eeq 
with $\kappa=\kappa_\nu$ or $\kappa=\kappa_p$ as in the following Lemma. The case $\kappa=\kappa_\nu$ with $\nu=0$ corresponds to the pure convergence case without rates. 
Using the interpolation inequality and Lemma 3.13 in \cite{HohageDiss} we immediately get the 
following result, that is crucial for convergence and convergence rates.
\begin{lem}\label{lemTalphabeta2}
Under Assumption \ref{ass_auinv} we have for $\alpha>0$ $\mu\in[0,\alpha]$,
 as well as any $\nu\in[0,1]$ and any $p>0$
\beq\label{Talphabeta_nu}
\alpha \left\|\left(\mathbf{T}^*\mathbf{T}+\left(\begin{array}{cc}\alpha I &0\\
0&\mu I \end{array}\right)\right)^{-1}
\kappa_\nu(\mathbf{T}^*\mathbf{T})\right\| 
\leq C_\nu \alpha^\nu
\eeq
\beq\label{Talphabeta_p}
\alpha \left\|\left(\mathbf{T}^*\mathbf{T}+\left(\begin{array}{cc}\alpha I &0\\
0&\mu I \end{array}\right)\right)^{-1}
\kappa_p(\frac{1}{e\|\mathbf{T}\|^2}\mathbf{T}^*\mathbf{T})\right\| 
\leq C_p f_p(\alpha)
\eeq
where
$\kappa_\nu(\lambda)=\lambda^\nu$, $\kappa_p(\lambda)=\ln(\frac{1}{\lambda})^{-p}$, $\lambda\in(0,1/e]$. 
\end{lem}
Therewith, the following convergence and convergence rates result 
with a priori chosen sequence $\alpha_k$ and stopping index $k_*$ follow directly along the lines 
of the proofs of Theorem 2.4 in \cite{BNS94} and Theorem 4.7 in \cite{HohageDiss}, see also 
Theorem 4.12 in \cite{KNSBuch}:
\begin{thm}
Let $\beta_k$ be a positive sequence decreasing monotonically to zero and satisfying 
\[
  \sup_{k\in\mathbb{N}}\frac{\beta_{k+1}}{\beta_{k}}<\infty\,,
\] 
and let $k_*=k_*(\delta)$ be chosen according to 
\[
k_*\to\infty\qquad\mbox{and}\qquad\eta\ge\delta\beta_{k_*}^{\frac12}\to 0\;\;\mbox{as}\;\;
 \delta\to 0\,
\]
and
\[
 \eta\beta_{k_*}^{-\nu-\frac12}\le\delta<\eta\beta_k^{-\nu-\frac12}\,,\qquad 0\le k<k_*\,,
\]
in case of \eqref{source_aao} with $\kappa(\lambda)=\lambda^\nu$, $\nu\in(0,1]$ 
or 
\[
 \frac{\eta}{\beta_{k_*}}\le\delta<\frac{\eta}{\beta_k}\,,\qquad 0\le k<k_*\,,
\]
in case of \eqref{source_aao} with $\kappa(\lambda)=\ln(\frac{1}{\lambda})^{-p}$, $p>0$ respectively.
\begin{enumerate}
\item If \eqref{source_aao} holds with $\kappa(\lambda)=\lambda^\nu$, $\nu\in[0,\frac12]$ 
or $\kappa(\lambda)=\ln(\frac{1}{\lambda})^{-p}$, $p>0$, we assume that 
\begin{eqnarray*}
 \mathbf{F}'(\tilde{q},\tilde{u}) &=& R((\tilde{q},\tilde{u}),(q,u))\mathbf{F}'(q,u)
 +Q((\tilde{q},\tilde{u}),(q,u)) \\
 \|I-R((\tilde{q},\tilde{u}),(q,u))\| &\le& c_R \\
\|Q((\tilde{q},\tilde{u}),(q,u))\| &\le& c_Q\|\mathbf{F}'(q^\dag,u^\dag)((\tilde{q},\tilde{u})
-(q,u))\| 
\end{eqnarray*}
for all $(q,u),(\tilde{q},\tilde{u})\in \mathcal{B}_{\rho}(q_0,u_0)$ and that
$\|(q^\dag,u^\dag)-(q_0,u_0)\|$, $\|(s,v)\|$, $\eta$, $\rho$, $c_R$ are sufficiently small.
\item If \eqref{source_aao} holds with $\kappa(\lambda)=\lambda^\nu$, $1/2\le\nu\le 1$, we assume that 
\[
	 \|\mathbf{F}'((\tilde{q},\tilde{u}))-\mathbf{F}'(q,u)\|\le L\|(\tilde{q},\tilde{u})-(q,u)\|
\]
for all $(q,u),(\tilde{q},\tilde{u})\in \mathcal{B}_{\rho}(q_0,u_0)$ and 
$\|(q^\dag,u^\dag)-(q_0,u_0)\|$, $\|(s,v)\|$, $\eta$, $\rho$ are sufficiently small.
\end{enumerate}
Then for $(\qksterndelta,\uksterndelta)$ defined by \eqref{eq_IRGNM_aao} (i.e., 
\eqref{eq_IRGNM_aao_var}), 
we obtain convergence $(\qksterndelta,\uksterndelta)\to(q^\dag,u^\dag)$ as $\delta\to0$ and 
convergence rates
\begin{equation}\label{rates0_aao}
\|\qhkstern -q^\dag\|^2+\|\uhkstern -u^\dag\|^2
\leq \frac{\bar{C}^2\delta^2}{\Theta^{-1} \left(\tfrac{\bar{C}}{2\|(s,v)\|}\delta\right)}
= 4\|(s,v)\|^2 \kappa^2(\Theta^{-1}(\tfrac{\bar{C}}{2\|(s,v)\|}\delta))
\end{equation}
where $\Theta(\lambda)\coloneqq\kappa(\lambda)\sqrt{\lambda}$. 
\end{thm}

\end{document}